\documentclass[12pt]{article}
\usepackage{amsmath,amssymb,amsthm,epsf, graphics}
\usepackage{graphicx}
\usepackage{xcolor}
\usepackage{mathtools}
\usepackage{float}
\usepackage{mathabx}
\usepackage{hyperref}
\usepackage{enumerate}
\usepackage{lipsum}
\usepackage{mwe}

\theoremstyle{definition}

\newtheorem{theo}{{\bf{Theorem}}}[section]
\newtheorem{defi}[theo]{{\bf Definition}}
\newtheorem{prop}[theo]{{\bf Proposition}}
\newtheorem{lem}[theo]{{\bf Lemma}}

\newtheorem{remark}[theo]{Remark}

\allowdisplaybreaks

\newcommand{\ZZ}{\mathbb{Z}}

\newcommand{\RR}{\mathbb{R}}

\newcommand{\TT}{\mathbb{T}}

\title{An improved bound on the Hausdorff dimension of sticky Kakeya sets in $\RR^4$}

\begin{document}

\author{
  Mukul Rai Choudhuri
}

\maketitle

\begin{abstract}
    Kakeya sets are compact subsets of $\RR^n$ that contain a unit line segment pointing in every direction. The Kakeya conjecture states that such sets must have Hausdorff dimension $n$. The property of stickiness was first discovered by Katz-Łaba-Tao in their 1999 breakthrough paper on the Kakeya problem. Then Wang-Zahl formalized the definition of a sticky Kakeya set, and proposed a special case of the Kakeya conjecture for such sets. Specifically this conjecture states that sticky Kakeya sets in $\RR^n$ have Hausdorff dimension $n$ and Wang-Zahl went on to prove the conjecture for $n=3$.
    
    A planebrush is a geometric object which is a higher dimensional analogue of Wolff's hairbrush. Using the planebrush argument, Katz-Zahl showed that Kakeya sets in $\RR^4$ have Hausdorff dimension at least 3.059. If we restrict our attention to sticky Kakeya sets, we can improve upon this bound by combining the planebrush result with additional stickiness property. To be precise, we will show in this paper that sticky Kakeya sets in $\RR^4$ have dimension at least 3.25.
\end{abstract}

\section{Introduction}

A \textit{Kakeya set} (sometimes called Besicovitch set) in $\RR^n$ is a compact set $E\subset \RR^n$ containing a unit line segment in every direction. In 1919, Besicovitch gave a construction of a Kakeya set of Lebesgue measure zero \cite{besi}. Interest in Kakeya sets grew even further after Fefferman gave a counter-example to the Fourier ball multiplier conjecture using Kakeya sets \cite{feff}. Followed by this, more surprising connections of Kakeya sets to problems in oscillatory integrals (harmonic analysis), the analysis of dispersive and wave equations (PDE), combinatorics and number theory were discovered \cite{bour,tao,wolff}, enhancing the importance of these sets.

Even though Kakeya sets with Lebesgue measure zero exist, a natural question then was: could these sets be big in some other sense. The Kakeya conjecture states that every Kakeya set in $\RR^n$ must have Hausdorff and Minkowski dimension $n$. For $n=2$, the Kakeya conjecture was proved by Davies in 1971 \cite{davies}. However, in higher dimensions only partial results exist. Drury in 1983 \cite{drury} and Christ, Duoandikoetxea, Rubio de Francia in 1986 \cite{christ} were able to prove that Kakeya sets for $n\geq 3$ had dimension at least $(n+1)/2$. In 1995, Wolff \cite{wolff1} improved this bound using the geometric ``hairbrush" argument to $(n+2)/2$, a higher dimensional analogue of the geometric ``bush" argument due to Bourgain which gave a pre-existing bound of $(n+1)/2$ for the dimension.

The next important juncture was when Bourgain \cite{bour} used additive combinatorics to make progress on the conjecture which also spurred many further developments. For $n=3$, Katz, Łaba and Tao in 2000 \cite{katzlabatao} proved that the upper Minkowski dimension of Kakeya sets is at least $5/2+c$ where $c$ is a small absolute constant. Later, in 2017, Katz and Zahl \cite{kzahl} proved the same for the Hausdorff dimension. A key ingredient in the argument by Katz, Łaba and Tao was the usage of a property called stickiness. Roughly speaking, a collection of $\delta$-tubes in $\RR^n$ whose coaxial lines point in $\delta$-separated directions is said to be sticky at the scale $\rho$ if the $\delta$-tubes arrange themselves into roughly $\rho^{1-n}$ $\rho$-tubes, each of which contain roughly $\rho^{n-1}/\delta^{n-1}$ $\delta$-tubes. Note that stickiness was initially defined as a property satisfied by certain collections of tubes and not subsets of Euclidean space. 

Motivated by the original definition of the property by Katz-Łaba-Tao \cite{katzlabatao}, Wang and Zahl created a definition of sticky Kakeya sets as a special subclass of Kakeya sets in \cite{wangzahl}. We would now like to present their definition. However, before we can define sticky Kakeya sets, we need some preliminary definitions and notation. Therefore, let $\mathcal{L}$ denote the set of (affine) lines in $\RR^n$. Given a line $l$, we denote the direction of $l$ as $u\in \RR\mathbb{P}^{n-1}$ and the unique point on $l$ such that the line joining it to the origin is perpendicular to $l$ as $p\in\RR^n$. We then equip $\mathcal{L}$ with the metric $d(l,l')=|p-p'|+\angle(u,u')$. 

We will also need the notion of packing dimension, so we define it here. Let $E$ be any subset of a metric space. We first define the packing pre-measure $\widetilde{\mathcal{P}}^{\theta}$ as
$$
\widetilde{\mathcal{P}}^{\theta}(E)=\lim_{\epsilon\downarrow 0}\bigg( \sup \sum_{j=1}^\infty (2r_j)^\theta \bigg)\, ,
$$
where the supremum is over all collections of disjoint open balls $\{ B(x_j,r_j)\}_{j=1}^\infty$ with centers in $E$ and radii $r_j<\epsilon$. Then we define the packing measure $\mathcal{P}^\theta$ as
$$
\mathcal{P}^\theta(E)=\inf\bigg\{\sum_{j=1}^\infty \widetilde{\mathcal{P}}^{\theta}(E_i) : E\subset \bigcup_{j=1}^\infty E_i \bigg\}.
$$
Finally we define the packing dimension of $E$ as 
\begin{equation}
    \dim_P(E)=\inf\{\theta: \mathcal{P}^\theta(E)=0\}.
\end{equation}
An introduction to packing dimension may be found in \cite{fgbook}. Since the direction map from $\mathcal{L}$ to $\RR\mathbb{P}^{n-1}$ is Lipschitz, if $L \subset \mathcal{L}$ contains a line in every direction then $\dim_P (L)\geq n-1$. Now we are ready to define sticky Kakeya sets. As mentioned earlier, this definition is due to Wang-Zahl \cite{wangzahl}.

\begin{defi}\label{stdef}
A compact set $K\subset \RR^n$ is called a \textit{sticky Kakeya set} if there is a set of lines $L$ with packing dimension $n-1$ that contains at least one line in each direction, so that $l\cap K$ contains a unit interval for each $l\in L$.
\end{defi}

Observe that in the above definition, there is no explicit appearance of the multi-scale self-similar property that was present in the original definition of stickiness by Katz-Łaba-Tao \cite{katzlabatao}. However, Wang-Zahl proceed to set up a discretization of sticky Kakeya sets which leads to multi-scale self-similar properties. We shall use the same setup and follow their discretization procedure in Section \ref{dmss} of this paper. We shall only state the results of this discretization, and then use the multi-scale self-similar properties in later sections of the paper. Proofs and full detail of the procedure can be found in Wang-Zahl \cite{wangzahl}. Although, to provide some intuition, we shall briefly sketch how the discretization process works in Section \ref{dmss}.  

The sticky Kakeya conjecture is the statement that every sticky Kakeya set in $\RR^n$ has Hausdorff and Minkowski dimension $n$. Since every sticky Kakeya set is in particular a Kakeya set, the sticky Kakeya conjecture is weaker than the Kakeya conjecture. More discussion on the connection between the Kakeya conjecture and the sticky Kakeya conjecture can be found in \cite{wangzahl}. The sticky Kakeya conjecture obviously holds for $n=2$ as the Kakeya conjecture itself is proven in this case \cite{davies}. Furthermore, in 2022, Wang and Zahl \cite{wangzahl} proved sticky Kakeya sets in $\RR^3$ have Hausdorff dimension $3$, so the sticky Kakeya conjecture is true for $n=3$ as well. In higher dimensions, the sticky Kakeya conjecture remains open, and in this paper we will be improving the bound for $n=4$. In particular, we prove the following.

\begin{theo}\label{main}
    Sticky Kakeya sets in $\RR^4$ have Hausdorff dimension at least 3.25.
\end{theo}

To further contextualize the above result, we now describe the pre-existing bounds for the Kakeya problem in $\RR^4$. Note that if we plug in $n=4$ for the bound given by the hairbrush argument of Wolff \cite{wolff1}, then we we get that Kakeya sets in $\RR^4$ have  Hausdorff dimension at least 3. Later on, the results of Guth-Zahl, Zahl and Katz-Rogers \cite{guthzahl,zahl1,katzrog} established that Kakeya sets in $\RR^4$ have Hausdorff dimension at least $3+1/40=3.025$. These three papers used algebraic/polynomial arguments that were inspired by Dvir's proof of the finite field analogue of the Kakeya conjecture \cite{dvir}.

The next milestone occurred in 2021, when Katz and Zahl improved the bound further to 3.059 using a new geometric argument called the planebrush method \cite{katzzahl}. A planebrush is a geometrical object which is a higher dimensional analogue of Wolff's hairbrush \cite{wolff1}. The planebrush argument gives a bound of $10/3$ for Kakeya sets in $\RR^4$ whose lines through a typical point concentrate into a 2-plane. This result of Katz-Zahl \cite{katzzahl} will be precisely stated as Theorem \ref{planebrush} in this paper. Guth and Zahl \cite{guthzahl} had previously established an estimate which gave a dimension bound of 3.25 for trilinear collections of tubes in $\RR^4$ (stated as Theorem \ref{trikak} in this paper). To prove that Kakeya sets in $\RR^4$ have Hausdorff dimension at least 3.059, Katz-Zahl \cite{katzzahl} combined both of the above results. The dimension bound of 3.059 that they obtained improves upon pre-existing bounds, but working with general Kakeya sets presents a lot of challenging technical issues. This is why the bound obtained is weaker than the bounds for plany and trilinear cases namely 10/3 and 3.25 respectively. 

However, if we restrict our attention to sticky Kakeya sets, the extra multi-scale structure we get eases the technical difficulties in running the argument, which allows us to obtain a better dimension bound of 3.25 as stated in Theorem \ref{main}. Note that this is the minimum of the bounds we get from the trilinear and plany cases separately. Now that we have contextualized our main result, we shall next briefly outline a sketch of its proof. As a final remark, we mention that some developments regarding the Kakeya conjecture have been skipped in the discussion here for the sake of brevity, and one can find more details in the survey articles by Wolff \cite{wolff} and Katz, Tao \cite{katztao}.

\subsection{A sketch of the proof}\label{pfske}

As we sketch the proof of Theorem \ref{main}, we shall simultaneously provide a brief overview of how this paper is organized. Firstly, in Section \ref{dmss}, we shall discuss the discretization of sticky Kakeya sets. The main results in this section are Proposition \ref{stidisc} and Proposition \ref{multiscaleprop}. Both of these propositions are due to Wang and Zahl \cite{wangzahl}. The latter gives the multi-scale self-similarity property which will be used many times in future sections. 

In Section \ref{mainlemsec}, we shall see that sticky Kakeya sets in $\RR^4$ split into broadly two cases. The first case is the trilinear case, as shown in Figure \ref{trifig}. In this case, most triples of tubes passing through a common point have direction vectors pointing in (quantitatively) linearly independent directions. This case is relatively more tractable. Applying the trilinear Kakeya bound due to Guth and Zahl \cite{guthzahl}, which will be stated as Theorem \ref{trikak} in this paper, we obtain the required dimension bound of 3.25 in this case. 

\begin{figure}[h]
        \centering
        \includegraphics[scale=0.4]{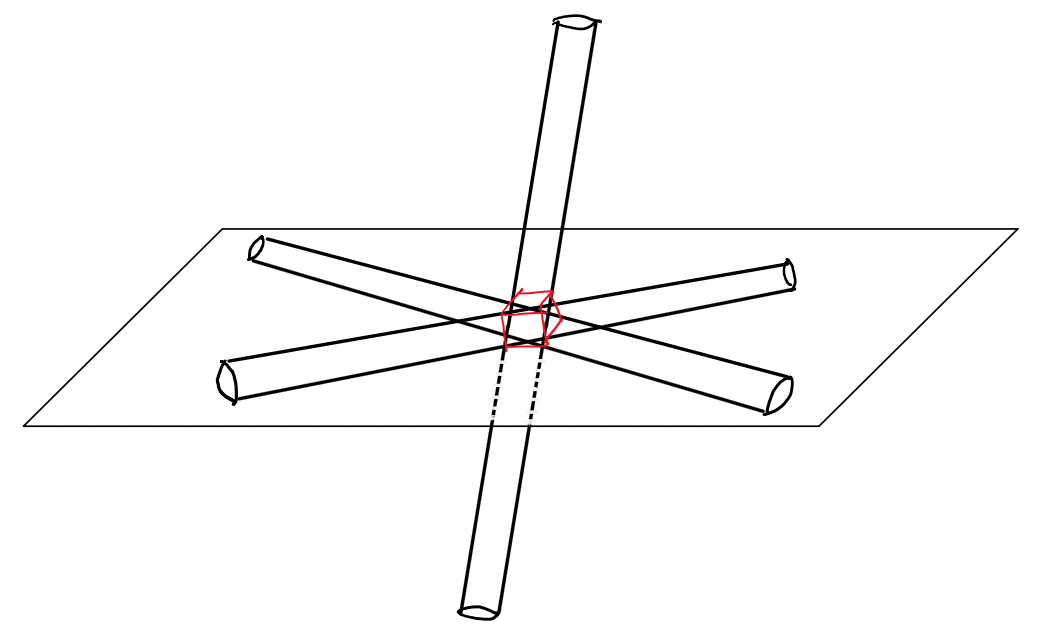}
        \caption{Trilinear case}
        \label{trifig}
\end{figure}

The second case is the negation of the first case, that is the situation when the tubes are not trilinear. This case is more challenging to deal with because we can not in general conclude that the $\delta$-tubes passing through any point lie in a $\delta$-neighborhood of a plane (this would have given us planyness in terms of Definition \ref{planydef}, further allowing us to directly establish a dimension bound of 10/3 using the planebrush bound due to Katz and Zahl \cite{katzzahl}). However, we can make the following weaker conclusion in this case: choosing $\rho=\delta^{1/N}$ for some appropriately chosen large integer $N$, we get that the $\delta$-tubes through any point lie in a $\rho$-neighborhood of a plane. We call this case the ``weakly" plany case, because although $\rho$ is much smaller than 1, it is much greater than $\delta$ (see Figure \ref{wkplnypic}). Using stickiness, we can show that the following holds at the scale $\rho$: the $\rho$-tubes through any point lie in a $\rho$-neighborhood of a plane (see Figure \ref{rhoplany}). Therefore, by Definition \ref{planydef}, we get that the collection at scale $\rho$ is plany. Thus, we can apply the planebrush bound, due to Katz and Zahl \cite{katzzahl}, at the coarser scale $\rho$. The planebrush bound will be stated in this paper as Theorem \ref{planebrush}.

Aside from the above considerations, we will also need to use induction on scales to get the required bound at the finer scale $\delta$. Suppose, for simplicity, that the multiplicity at scale $\delta$, $\rho$ and within every $\rho$-tube is constant and let these be equal to $\mu_\delta$, $\mu_{\text{coarse}}$ and $\mu_{\text{fine}}$. Then using stickiness, we can get the heuristic relation
\begin{equation}\label{multieq}
\mu_\delta=\mu_{\text{fine}}\mu_{\text{coarse}}.
\end{equation}
Again, this should not be taken literally. It is true only in a rough sense, rigorous expressions of this will occur in Proposition \ref{multiscaleprop} as upper bounds (see Remark \ref{multiscalermk}). This relationship will be used multiple times in later sections of this paper: to establish regularity, and to connect across scales when we do induction. The induction on scales will be done rigorously in Section \ref{pfsec}, where we will complete the proof of Theorem \ref{main} and conclude the paper.

\begin{figure}[h]
    \centering
    \begin{minipage}{0.45\textwidth}
        \centering
        \includegraphics[scale=0.4]{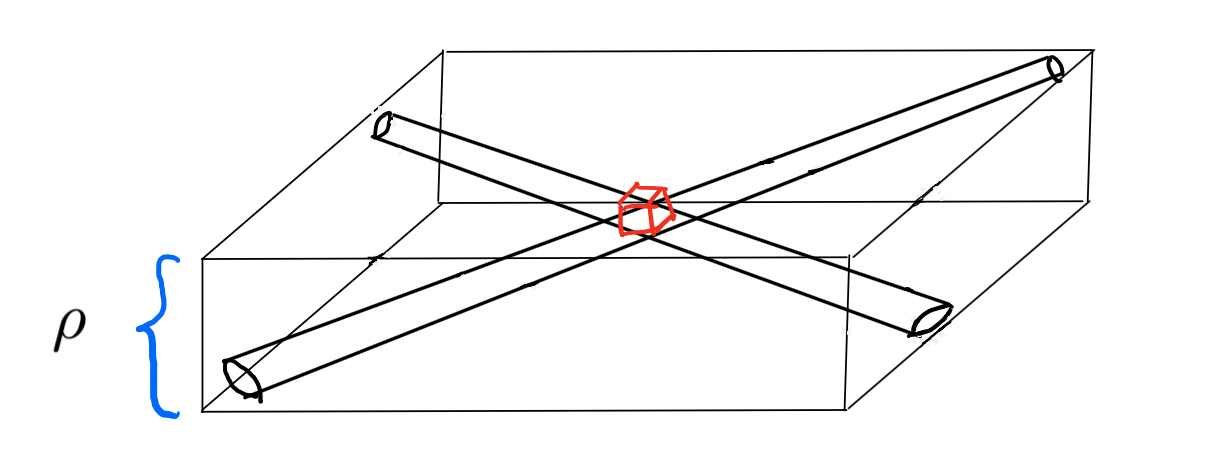} % first figure itself
        \caption{``Weakly" plany case. Only two $\delta$-tubes shown for clarity.}
        \label{wkplnypic}
    \end{minipage}\hfill
    \begin{minipage}{0.45\textwidth}
        \centering
        \includegraphics[scale=0.4]{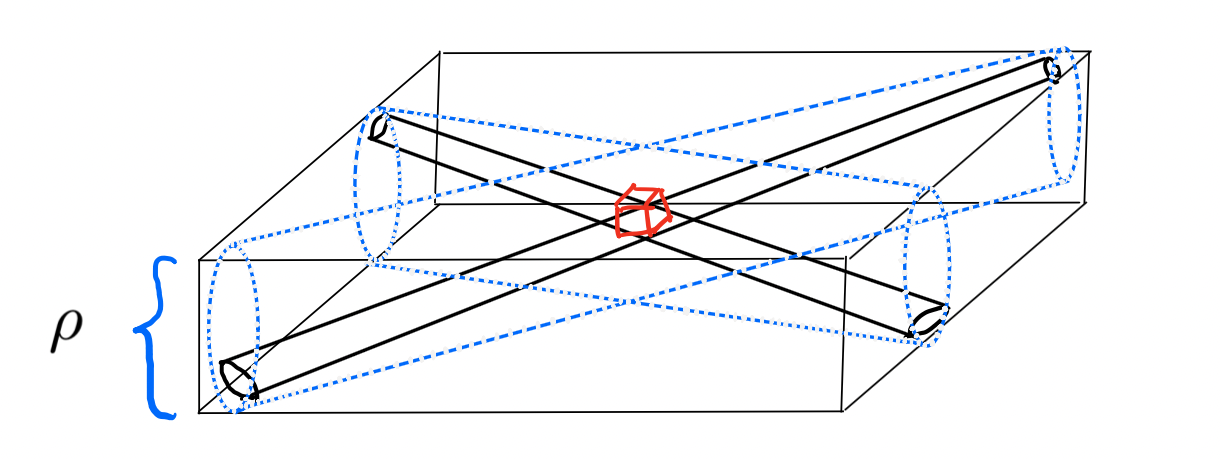} % second figure itself
        \caption{Tubes at scale $\rho$ are plany in this case.}
        \label{rhoplany}
    \end{minipage}
\end{figure}

As a final note, we discuss the most technically difficult juncture of this paper. Namely, the splitting of the sticky Kakeya set into the trilinear and plany cases, which is done in Lemma \ref{mainlem}. Multi-scale self-similarity coming from stickiness (Proposition \ref{multiscaleprop}) plays a key role in proving the lemma. By combining multi-scale structure with the trilinear-plany dichotomy, Lemma \ref{mainlem} plays a crucial role when we finally do induction on scales in Section \ref{pfsec}. Katz and Zahl \cite{katzzahl} also had a lemma which split between trilinear and plany cases, but for general Kakeya sets. While there are some similarities between the two lemmas, there are also important differences. The lemma in Katz-Zahl had weaker conclusions than Lemma \ref{mainlem}, because they did not have the additional hypothesis of stickiness. This in turn led to their weaker general bound of 3.059.

\subsection{Acknowledgement}

The author would like to thank Izabella Łaba and Joshua Zahl for many helpful conversations and suggestions during the preparation of this manuscript. 

\section{Framework}\label{framsec}

Since the procedure of discretizing sticky Kakeya sets was already introduced and fleshed out in Wang-Zahl \cite{wangzahl}, we shall follow their work. Specifically, we shall define lines, cubes, shadings, refinements the way they did. This will then allow us to import their discretization results (Proposition \ref{stidisc} and Proposition \ref{multiscaleprop}) directly. The setup in Katz-Zahl \cite{katzzahl} only has some minor differences with the way things are defined in Wang-Zahl \cite{wangzahl}. Therefore, by following the definitions of Wang-Zahl \cite{wangzahl}, we can import results from Katz-Zahl \cite{katzzahl} as well without trouble. Specifically Theorem \ref{planebrush}, which is the planebrush bound, will be quoted from Katz-Zahl \cite{katzzahl}. Now let us get started with the required definitions. 

We define $\mathcal{L}_n$ to be the set of lines in $\RR^n$ of the form $(p,0)+\RR v$, where $p\in [-\frac{1}{n},\frac{1}{n}]^{n-1}$ and $v\in S^{n-1}$ has final coordinate $v_n\geq 1/2$. We define $d(l,l')=|p-p'|+\angle(v,v')$. We define the measure $\lambda_n$ to be the product of $(n-1)$ dimensional Lebesgue measure and normalized surface measure on $S^{n-1}$ (we denote the latter by $\nu_{n-1}$). With these definitions of distance and measure, the packing dimension of a set $L\subset \mathcal{L}_n$ agrees with its upper modified box dimension. 
\begin{defi} (Tubes).
    We define the $\delta$-\textit{tube} $T$ with coaxial line $l$ to be $N_{2n\delta}(l)\cap [-1,1]^n$ where $N_{2n\delta}(l)$ represents the $2n\delta$-neighborhood of the line $l$. We say the $\delta$-tubes $T,T'$ are \textit{essentially identical} if $d(l,l')\leq\delta$ where $l,l'$ are the coaxial lines for $T,T'$ respectively, otherwise $T$ and $T'$ are \textit{essentially distinct}. We say that $T,T'$ are \textit{essentially parallel} if $\angle(v,v')\leq \delta$, where $v,v'\in S^{n-1}$ represent the directions of the coaxial lines of the tubes $T,T'$ respectively.
\end{defi} 
This definition is a bit nonstandard, since we are using the $2n\delta$-neighborhood instead of the $\delta$-neighborhood of $l$. This is done so that a cube of side-length $\delta$ intersecting $l$ will be contained in the associated tube $T$. 

\begin{defi} (Cubes and shadings).
    A $\delta$-\textit{cube} is a set of the form $Q=[0,\delta)^n+p$, where $p\in (\delta\ZZ)^n$. A \textit{shading} of a $\delta$-tube $T$ is a set $Y(T)\subset T$ that is a union of $\delta$-cubes.  If $\TT$ is a collection of $\delta$-tubes and for each $T\in \TT$, we have the associated shading $Y(T)$; then we refer to the pair $(\TT,Y)_\delta$ as a set of tubes and their associated shading.
\end{defi}

Note that sometimes shadings are defined as arbitrary subsets of tubes. However, both the papers of Wang-Zahl \cite{wangzahl} and Katz-Zahl \cite{katzzahl} use discrete shadings consisting of cubes. This offers advantages in certain situations and since we are using many results from these two papers, we also use discrete version of shadings with cubes as in the definition above.

 If $(\TT,Y)_\delta$ is a set of tubes and their associated shading, we write $E_\TT$ to denote the set $\bigcup_{T\in \TT} Y(T)$. We write $\mathcal{Q}(Y)$ to denote the set of $\delta$-cubes that are in at least one shading $Y(T)$ for some $T\in \TT$. For each $\delta$-cube $Q$, we define $\TT_Y(Q)=\{T\in \TT: Q\subset Y(T)\}$. The average density and multiplicity of a shaded collection of tubes $(\TT,Y)$ is denoted as 
$$
\lambda_Y=\frac{1}{\#\TT}\sum_{T\in \TT} \frac{|Y(T)|}{|T|}\, \textrm{  and 
 }\,\mu_Y=\frac{\sum_{Q\in \mathcal{Q}(Y)} \#\TT_Y(Q)}{\#\mathcal{Q}(Y)} \,\,\textrm{  respectively.}
$$ 

Next, we would like to define refinements. However, before we do that, we require a little bit of notation. For two quantities $A,B$, we will denote by $A\lesssim B$ the fact that the inequality $A\leq C B$ holds with an implicit constant $C$ typically depending on the dimension $n$ of the Euclidean space we are working in. It could also depend on other parameters which will be clear from context or stated explicitly. We will write $A \sim B$ if $A\lesssim B$ and $B\lesssim A$.

\begin{defi}(Refinements).
    We say that $(\TT',Y')_\delta$ is a $t$-\textit{refinement} of $(\TT,Y)_\delta$ if $\TT'\subset \TT$ and $Y'(T)\subset Y(T)$ for all $T\in \TT'$, and if we have 
    \begin{equation}\label{refineineq1}
        \sum_{T\in \TT'} |Y'(T)|\geq t \sum_{T\in \TT} |Y(T)|.
    \end{equation} Furthermore, if we have 
    \begin{equation}\label{refineineq2}
        \sum_{T\in \TT'} |Y'(T)|\gtrsim |\log \delta|^{-C} \sum_{T\in \TT} |Y(T)|,
    \end{equation}
    where $C\geq 0$ is a constant that depends only on $n$, then we just call it as a refinement.
\end{defi} 

\begin{remark}
    We shall sometimes refer to the quantity $\sum_{T\in \TT} |Y(T)|$ as the ``mass" of the shading.  Observe that if $(\TT',Y')_\delta$ is a $t$-refinement of $(\TT,Y)_\delta$, then $\lambda_{Y'}\geq t\lambda_Y$. Therefore, refinements lower both mass and density by a factor of at most $t$. Note that it is possible that the new density $\lambda_{Y'}$ could be much larger than the original density $\lambda_Y$. On the other hand, the mass of the new shading can never exceed the mass of the original shading. 
\end{remark}

\begin{lem}\label{constaref}
    Let $(\TT,Y)_\delta$ be a collection of $\delta$-tubes and their associated shading. Further assume that $\#\TT\lesssim \delta^{-100n}$ (this will be satisfied by any collection one meets in practice). Then
    \begin{enumerate}
        \item  There exists a refinement $(\TT',Y')_\delta$ of $(\TT,Y)_\delta$ such that $|Y'(T)|/|T|\sim \lambda_{Y'}$ for all $T\in \TT'$ i.e the density is constant across all tubes.
        \item There exists a (in general, different from previous) refinement $(\TT'',Y'')_\delta$ of $(\TT,Y)_\delta$ such that $\#{\TT''}_{Y''}(Q)\sim \mu_{Y''}$ for all $Q\in \mathcal{Q}(Y'')$.
    \end{enumerate}
\end{lem}
\begin{proof}
    For (a) we dyadically partition $\TT$ into sub-collections $\{\TT_i\}$ where $\TT_i$ is defined for each $i$ as the set of tubes $T\in\TT$ such that we have $|Y(T)|/|T|\in (2^{-i},2^{-i+1}]$. Since the volume of a $\delta$-cube is $\delta^n$, for any tube which has a non empty shading we have $|Y(T)|/|T|\gtrsim \delta$. Hence the index $i$ runs from $1$ until $C_n\log \delta^{-1}$. By pigeonholing there exists an $i$ such that 
    $$
    \sum_{T\in \TT_i}|Y(T)|\gtrsim |\log \delta|^{-1} \sum_{T\in \TT}|Y(T)|.
    $$
    Therefore, if we let $\TT'=\TT_i$ and $Y'=Y$ for these tubes, then by \eqref{refineineq2}, we see that $(\TT',Y')_\delta$ is a refinement with the desired constant density property.

    For (b), we first dyadically partition the cubes $\mathcal{Q}(Y)$ into sub-collections $\{\mathcal{Q}_i\}$ where $\mathcal{Q}_i$ is defined as the set of cubes $Q\in \mathcal{Q}(Y)$ such that $\#\TT_Y(Q)\in [2^{i-1},2^i)$. Since $\#\TT\lesssim \delta^{-100n}$, the index $i$ runs from 1 until $C_n\log \delta^{-1}$. By pigeonholing there exists $i$ such that
    \begin{equation}\label{refineineqlem1}
        \sum_{Q\in \mathcal{Q}_i} \#\TT_Y(Q) \gtrsim |\log \delta|^{-1} \sum_{Q\in \mathcal{Q}(Y)} \#\TT_Y(Q).
    \end{equation}
    We would like to convert the above inequality into one of the form \eqref{refineineq2}. For this we need to convert a sum involving multiplicities into one that involves volumes of shadings. This can be done as follows. First define the incidence function $f:\mathcal{Q}(Y)\times \TT\to \RR$ as
    $$
    f(Q,T) = \begin{cases}
1 &\text{if $Q\subset Y(T)$}\\
0 &\text{otherwise.}
\end{cases}
    $$
    Now by interchanging the following finite double-sum, we get that
    \begin{equation}\label{refineineqlem2}
        \delta^n\sum_{Q\in \mathcal{Q}(Y)} \#\TT_Y(Q)=\delta^n\sum_{Q\in \mathcal{Q}(Y)} \sum_{T\in \TT} f(Q,T)=\sum_{T\in \TT} \bigg(\delta^n\sum_{Q\in \mathcal{Q}(Y)} f(Q,T)\bigg)=\sum_{T\in \TT} |Y(T)|.
    \end{equation}
    The above equality connecting multiplicities and volumes is worth noting as it will show up multiple times in various junctures of this paper. Observe that it also shows that a shading is essentially a collection of cube-tube incidence pairs. Now we define $(\TT'',Y'')_\delta$ to be shaded collection of tubes one gets by throwing away all the cubes outside of $\mathcal{Q}_i$ from $\mathcal{Q}(Y)$. Moreover, for every $Q\in \mathcal{Q}_i$, we keep every tube from $\TT_Y(Q)$ in the new collection. Now combining \eqref{refineineqlem1} and \eqref{refineineqlem2}, we get that $(\TT'',Y'')_\delta$ is a refinement of $(\TT,Y)_\delta$ completing the proof.    
\end{proof}

\begin{defi} (Covers)
    Let $T$ be a $\delta$-tube and $\Tilde{T}$ be a $\rho$-tube, with $\delta\leq \rho$. We say that $\Tilde{T}$ \textit{covers} $T$ if their respective coaxial lines satisfy $d(l,\Tilde{l})\leq \rho/2$. If $\TT$ is a collection of $\delta$-tubes and $\Tilde{\TT}$ is a collection of $\rho$-tubes, we say that $\Tilde{\TT}$ covers $\TT$ if every tube in $\TT$ is covered by some tube in $\Tilde{\TT}$. If this is the case, we write $\TT[\Tilde{T}]$ to denote the set of tubes in $\TT$ that are covered by $\Tilde{T}$. Similarly we say $(\Tilde{\TT},\Tilde{Y})_\rho$ covers $(\TT,Y)_\delta$ if $\Tilde{\TT}$ covers $\TT$ and for each $\Tilde{T}\in \Tilde{\TT}$ and each $T\in \TT[\Tilde{T}]$, we have $Y(T)\subset \Tilde{Y}(\Tilde{T})$.
\end{defi}

Sometimes we will encounter a situation where $(\Tilde{\TT},\Tilde{Y})_\rho$ covers $(\TT,Y)_\delta$ and we do a refinement on $(\Tilde{\TT},\Tilde{Y})_\rho$. However, this may damage the covering property. To rectify this situation, we make the following definition.

\begin{defi}\label{canondef}
    (Canonical refinements). Let $(\TT,Y)_\delta$ and $(\Tilde{\TT},\Tilde{Y})_\rho$ be collections of $\delta$-tubes and $\rho$-tubes respectively with their associated shading. Further suppose that $(\Tilde{\TT},\Tilde{Y})_\rho$ covers $(\TT,Y)_\delta$ and that $(\Tilde{\TT}',\Tilde{Y}')_\rho$ is a $t$-refinement of $(\Tilde{\TT},\Tilde{Y})_\rho$. Then we define the \textit{canonical refinement} $(\TT',Y')_\delta$ to be the largest possible refinement of $(\TT,Y)_\delta$ such that $(\Tilde{\TT}',\Tilde{Y}')_\rho$ covers $(\TT,Y)_\delta$.
\end{defi}

\begin{remark}\label{canonrmk}
If $(\Tilde{\TT}',\Tilde{Y}')_\rho$ is a $t$-refinement of $(\Tilde{\TT},\Tilde{Y})_\rho$ and $(\TT',Y')_\delta$ is the canonical refinement of $(\TT,Y)_\delta$, then in general we cannot say that $(\TT',Y')_\delta$ is also a $t$-refinement of $(\TT,Y)_\delta$. However, if we have some regularity property of how the $\delta$-cubes are arranged inside the larger $\rho$-cubes then we can say something. Specifically, suppose that each tube $T\in \TT'$ is covered by a single tube from $\Tilde{\TT}$ and further suppose that there exists a constant $w>0$ such that 
\begin{equation}\label{canonical}
    \sum_{T\in \TT[\Tilde{T}]} |Y(T)\cap \Tilde{Q}|\sim w
\end{equation}
for each $\Tilde{T}\in \Tilde{\TT}$ and $\Tilde{Q}\in \Tilde{Y}(\Tilde{T})$. Note that $w$ here does not depend on choice of $\Tilde{T}$ and $\Tilde{Q}$, and we don't impose any condition on its size. Then we get that $(\TT',Y')_\delta$ is a $t'$-refinement of $(\TT,Y)_\delta$ with $t'\gtrsim t$. This can be seen as follows. The fact that $(\Tilde{\TT}',\Tilde{Y}')_\rho$ is a $t$-refinement of $(\Tilde{\TT},\Tilde{Y})_\rho$ means that we are keeping a $t$-fraction of incidence pairs $(\Tilde{Q},\Tilde{T})$ (see \eqref{refineineqlem2}) at the larger scale. Moreover, because of \eqref{canonical}, every incidence pair $(\Tilde{Q},\Tilde{T})$ at the larger scale has the same amount of ``$\delta$-mass" $\sim w$ inside of it. The reason we mention this here is that these conditions appear as Properties (i) and (vi) in Proposition \ref{multiscaleprop}. Therefore, in that situation there, canonical refinements are well behaved by this remark. This will come in handy in the proof of Lemma \ref{mainlem}.
\end{remark}

\section{Discretization and multi-scale structure}\label{dmss}
In this section we will be discretizing sticky Kakeya sets and realizing their multi-scale self-similarity properties. The framework and results are due to Wang and Zahl \cite{wangzahl}. We shall borrow definitions and propositions from them directly. First we define a quantity called $\sigma_n$, that will allow us to formulate a discretized version of Theorem \ref{main}. 
\begin{defi}\label{31}
    For $s,t,\delta\in (0,1)$, define
    \begin{equation}\label{presigmadef1}
        M(s,t,\delta)=\inf_{\TT,Y} \bigg| \bigcup_{T\in \TT} Y(T)\bigg|.
    \end{equation}
    Here the infimum is taken over all pairs $(\TT,Y)_\delta$ with the following properties.
    \begin{enumerate}
        \item The tubes in $\TT$ are essentially distinct.
        \item For each $\delta\leq \rho\leq 1$, $\TT$ can be covered by a set of $\rho$-tubes, with the property that at most $\delta^{-t}$ $\rho$-tubes are essentially parallel to any given $\rho$-tube.
        \item $\sum_{T\in \TT} |Y(T)|\geq \delta^s$.
    \end{enumerate}
\end{defi}

Note that part (b) roughly says that $\TT$ is a collection of $\delta$-tubes pointing in $\delta$ separated directions and that further it is covered by a collection of $\rho$-tubes pointing in $\rho$-separated directions. This is similar to the way stickiness was originally defined for collections of tubes in Katz, Łaba, Tao \cite{katzlabatao}. Next, let
\begin{equation}\label{presigmadef2}
    N(s,t)=\limsup_{\delta\to 0^+}\frac{\log M(s,t,\delta)}{\log \delta}.
\end{equation}
Then we define $\sigma_n$ as
\begin{equation}\label{sigmadef}
    \sigma_n=\lim_{(s,t)\to (0,0)}N(s,t)=\inf_{(s,t)\in (0,1)^2} N(s,t).
\end{equation}
Having defined the quantity $\sigma_n$, we would now like to connect it with Theorem \ref{main}. This will give us a discretized formulation of the theorem. Before we get to that, recall that sticky Kakeya sets are those Kakeya sets whose lines have packing dimension $n-1$ (Definition \ref{stdef}). We now relate the Hausdorff dimension of sticky Kakeya sets with the quantity $\sigma_n$ in the following proposition. This is Proposition 2.1 from Wang-Zahl \cite{wangzahl}.

\begin{prop}\label{stidisc}
Let K be a sticky Kakeya set. Then $\dim_H K \geq n- \sigma_n$.
\end{prop}

The full proof of the above proposition can be found in Wang-Zahl \cite{wangzahl} and we shall not reproduce it here. However, we shall give an informal idea of how discretizing sticky Kakeya sets gives us collections of tubes that satisfy Properties (a), (b), (c) from Definition \ref{31}. 

For the sake of convenience we will be working in $\RR^2$, but the sketch presented can be generalized to any dimension. Again for convenience, let us characterize/make an alternative definition of a Kakeya set in $\RR^2$. First we define a function $\phi:[0,1]\to [0,1]$ which is to assign a basepoint to every direction and hence called a Kakeya map. Associated to this Kakeya map $\phi$, we define the Kakeya set $K_\phi=\{(\phi(v),0)+t(v,1):t,v\in [0,1]\}$. In this setup, stickiness in terms of Definition \ref{stdef} means that the graph of the function $\phi$ (which is a subset of $[0,1]^2$) has Assouad dimension 1. Therefore, at the scale $\delta$, the graph can be covered by $1/\delta$ many $\delta$ squares and the same thing holds for $\rho$. Let us represent these two scales pictorially.

\begin{figure}[h]
\includegraphics[width=8cm, height=5cm]{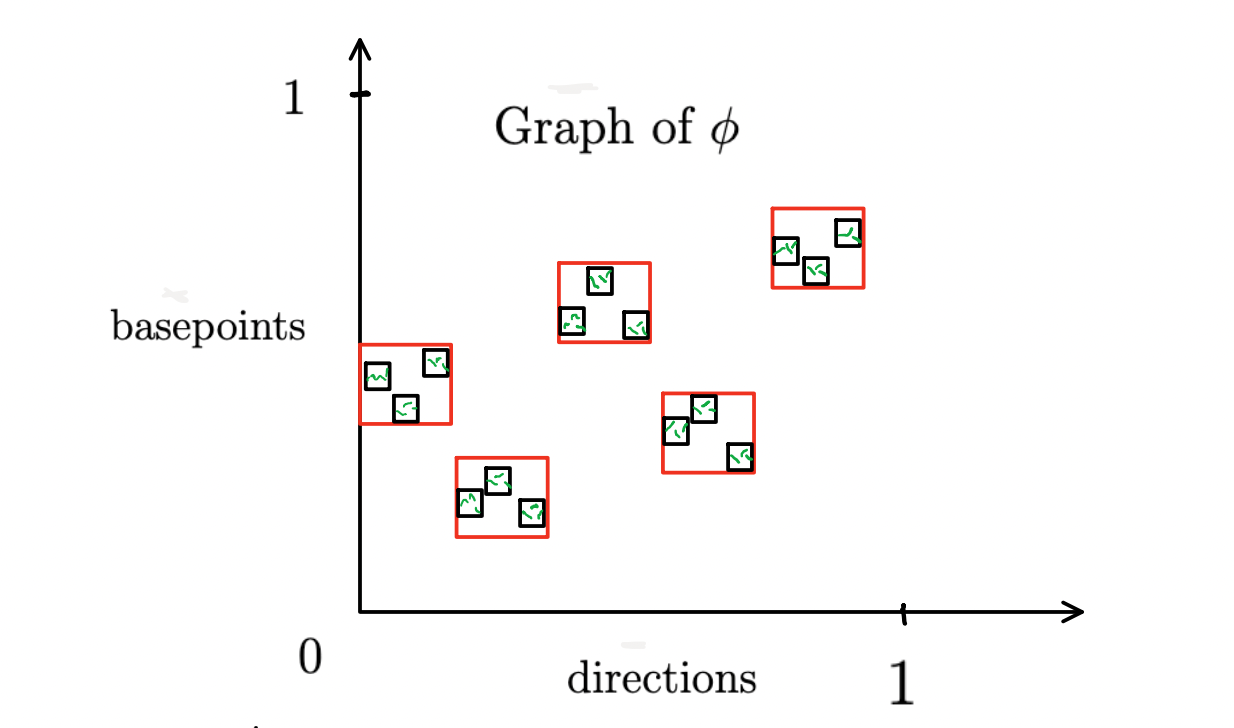}
\centering
\caption{Discretizing sticky Kakeya sets}
\label{stipic}
\end{figure}

We will have the squares appearing consecutively as shown in the picture, with the $\delta$ squares nested inside the $\rho$ squares. Each $\rho$ square has $\rho/\delta$ many squares inside of it. The $\delta$ and $\rho$ squares in the above graph represent $\delta$ and $\rho$ tubes in $\RR^2$ respectively. Therefore, this nesting in the graphical representation means that the collection of $\delta$-tubes pointing in $\delta$ separated directions are covered by the collection of $\rho$-tubes pointing in $\rho$-separated directions. Thus, we have traveled from the dimension of the graph of the Kakeya map as in Definition \ref{stdef} to a discretized statement about tubes with multi-scale structure as in Definition \ref{31}. We have now connected the two notions of stickiness.

Recall that Theorem \ref{main}, which is our main goal, states that sticky Kakeya sets in $\RR^4$ have Hausdorff dimension at least 3.25. Therefore, by virtue of Proposition \ref{stidisc}, to prove Theorem \ref{main}, it suffices to show the following.
\begin{theo}\label{sti2}
    $\sigma_4\leq 3/4$.
\end{theo}
Thus, our aim now becomes to prove the above statement. However, recall that the quantity $\sigma_n$ is defined using the volumes of shadings of tubes. Hence, we would like to make a further reduction to a statement concerning the volumes of shadings of tubes. This will be easier for us to work with. For this purpose, we need to introduce extremal collections of tubes which are a special case of the collections of tubes we saw in Definition \ref{31}.
\begin{defi}
    Let $\eta,\delta>0$. We say a pair $(\TT,Y)_\delta$ is $\eta$-\textit{extremal} if it satisfies Items (a), (b) and (c) from Definition \ref{31} with $s=t=\eta$ and furthermore 
    \begin{equation}\label{extrprop}
        \bigg|\bigcup_{T\in \TT} Y(T)\bigg|\leq \delta^{\sigma_n-\eta}.
    \end{equation}
\end{defi}

\begin{remark}\label{smolrmk}
    Suppose that $(\TT,Y)_\delta$ is $\eta$-extremal and that $(\TT',Y')_\delta$ is a refinement of $(\TT,Y)_\delta$. After doing such a refinement, all upper bounds are preserved. Therefore, $(\TT',Y')_\delta$ satisfies Items (a),(b) from Definition \ref{31} with $t=\eta$, and also from \eqref{extrprop}, we get that $|\bigcup_{T\in \TT'} Y'(T)|\leq \delta^{\sigma_n-\eta}$. The only inequality which might fail is (c) of Definition \ref{31}. However, since $(\TT',Y')_\delta$ is a refinement of $(\TT,Y)_\delta$, we get that 
    $$
    \sum_{T\in \TT'}|Y'(T)|\geq |\log \delta|^{-C} \sum_{T\in \TT} |Y(T)|\geq \delta^{2\eta}.
    $$ 
    Therefore, $(\TT',Y')_\delta$ is $2\eta$-extremal. We shall use this little fact often in our arguments.
\end{remark}

\begin{remark}\label{hugermk}
Note that if $(\TT,Y)_\delta$ is $\eta$-extremal, then 
\begin{equation}\label{numtubes}
    \delta^{1-n+\eta}\lesssim \#\TT \lesssim \delta^{1-n-\eta}
\end{equation}
This follows from Properties (b) and (c) of Definition \ref{31}. Combing the above inequality again with Property (c) of Definition \ref{31}, we obtain the following useful inequality on the density of the shading 
\begin{equation}\label{densityext}
    \lambda_Y=\frac{1}{\#\TT}\sum_{T\in \TT} \frac{|Y(T)|}{|T|}\gtrsim \delta^{2\eta}.
\end{equation}
Therefore, we see that extremal collections are an almost full maximal collection of $\delta$-tubes pointing in $\delta$-separated directions with some subpolynomial losses. As for the average multiplicity, we have
\begin{equation}\label{avgmulti1}
    \mu_Y=\frac{\sum_{Q\in \mathcal{Q}(Y)} \#\TT_Y(Q)}{\#\mathcal{Q}(Y)}=\frac{\sum_{T\in \TT}|Y(T)|}{\delta^n \#\mathcal{Q}(Y)}=\frac{\sum_{T\in \TT}|Y(T)|}{|\bigcup_{T\in \TT} Y(T)|}.
\end{equation}
Here we used \eqref{refineineqlem2}. Next combining \eqref{extrprop} with Property (c) of Definition \ref{31}, we get the following lower bound on average multiplicity
\begin{equation}
    \mu_Y=\frac{\sum_{T\in \TT}|Y(T)|}{|\bigcup_{T\in \TT} Y(T)|}\gtrsim \frac{\delta^\eta}{\delta^{\sigma_n-\eta}}=\delta^{\sigma_n+2\eta}.
\end{equation}
We can also achieve an upper bound for the average multiplicity. However, the order of quantifiers is more subtle for this. Hence, we make a separate lemma for it.
\end{remark}

\begin{lem}\label{hugermk2}
    For every $\epsilon$, there exists an $\eta>0$ and $\delta_0>0$ so that the following holds for all $\delta\in (0,\delta_0]$. Let $(\TT,Y)_\delta$ be an $\eta$-extremal collection. Then we have
\begin{equation}\label{lbuni}
    \bigg|\bigcup_{T\in \TT} Y(T)\bigg|\geq \delta^{\sigma_n+\epsilon/2}.
\end{equation}
Further, we also have $\mu_Y\lesssim \delta^{-\sigma_n-\epsilon}$.
\end{lem}

\begin{proof}
    Using the definition of $\sigma_n$ (equation \ref{sigmadef}), we get that for every $\epsilon>0$ there exists $\eta>0$ such that $N(\eta,\eta)<\epsilon/4$. We can also assume that $\eta<\epsilon/2$. Then by choosing $\delta$ small enough, \eqref{presigmadef1} and \eqref{presigmadef2} imply that for every $\eta$-extremal collection $(\TT,Y)_\delta$, we have \eqref{lbuni}. It remains to show the upper bound on the average multiplicity. 
    
    Note that by \eqref{numtubes}, we have $\#\TT \lesssim \delta^{1-n-\eta}$, which further implies that $\sum_{T\in \TT}|Y(T)|\lesssim \delta^{-\eta}\lesssim \delta^{-\epsilon/2}$. Combining this, the above inequality and \eqref{avgmulti1}, we get
    \begin{equation}
    \mu_Y=\frac{\sum_{T\in \TT}|Y(T)|}{|\bigcup_{T\in \TT} Y(T)|}\lesssim \frac{\delta^{-\epsilon/2}}{\delta^{\sigma_n+\epsilon/2}}=\delta^{-\sigma_n-\epsilon}.
    \end{equation}
\end{proof}

Therefore, from the above lemma and remark, we see that $\eta$-extremal collections have a lot of nice properties. Ignoring subpolynomial losses for a moment, we can roughly say extremal collections $(\TT,Y)_\delta$ have $\delta$-tubes pointing in $\delta$-separated directions and have size $\#\TT=\delta^{1-n}$, mass $\sum_{T\in \TT}|Y(T)|=1$, density $\lambda_Y=1$, volume $|\bigcup_{T\in \TT} Y(T)|=\delta^{\sigma_n}$, multiplicity $\mu_Y=\delta^{-\sigma_n}$. Again these are not to be taken literally as they are only roughly true but this intuition is useful to have. We will make frequent use of the inequalities stated above to formalize this intuition when needed.

Using the definition of $\sigma_n$ (equation \eqref{sigmadef}), we can always find extremal collections for small $\delta$. More formally, we have the following.
\begin{lem}\label{extrexist}
    Let $\eta,\delta_0>0$. Then there exists an $\eta$-extremal collection of tubes $(\TT,Y)_\delta$ for some $\delta\in (0,\delta_0]$.
\end{lem}

We also create some notation that is quite handy while working with extremal collections of tubes. In particular, it will be heavily used in Section \ref{pfsec} when we do induction on scales.

\begin{defi}
    We define $V(\eta,\delta)=\inf_{\TT,Y} \big| \bigcup_{T\in \TT} Y(T)\big|$ where we are taking the infimum over all $\eta$-extremal collections $(\TT,Y)_\delta$.
\end{defi}

Having defined extremal collections of tubes and discussed their properties, we are now ready to reduce Theorem \ref{sti2} to a statement involving unions of tubes.

\begin{theo}\label{goal}
    Let $n=4$. Then, for every $\epsilon>0$, there exists a constant $c_\epsilon>0$ and $\eta>0$ such that $V(\eta, \delta)\geq c_\epsilon \delta^{\frac{3}{4}+\epsilon}$. That is to say whenever $(\TT,Y)_\delta$ is an $\eta$-extremal collection of tubes in $\RR^4$, we have
    \begin{equation}\label{mainineq}
        \bigg|\bigcup_{T\in \TT} Y(T)\bigg|\geq c_{\epsilon}\delta^{\frac{3}{4}+\epsilon}.
    \end{equation}
\end{theo}

Now let us see how the above result implies that Theorem \ref{sti2} is true and therefore Theorem \ref{main} as well.

\begin{proof} [Proof of Theorem \ref{sti2} assuming Theorem \ref{goal}.] Because of our hypothesis (where we choose $\eta<\epsilon$) and due to Lemma \ref{extrexist}, we get that
$$
c_\epsilon \delta^{\frac{3}{4}+\epsilon} \leq \bigg|\bigcup_{T\in \TT} Y(T)\bigg| \leq \delta^{\sigma_4-\eta}
$$
for
an infinite sequence of choices of $\delta$ tending  to $0$. Therefore, $\frac{3}{4}+\epsilon\geq \sigma_4-\eta$. This implies $\sigma_4\leq \frac{3}{4}+\epsilon+\eta\leq \frac{3}{4}+2\epsilon$. Now let $\epsilon\to 0$, then we get $\sigma_4\leq \frac{3}{4}$. 
\end{proof}

Therefore, now our goal becomes to prove Theorem \ref{goal} and this is what we will be focusing on proving from here onwards. After doing more build up work in the following sections, Theorem \ref{goal} will be finally proved in Section \ref{pfsec}, establishing Theorem \ref{main} consequently and completing our objective. Lastly, we need a stronger version of multi-scale self-similarity than Property (b) of Definition \ref{31} for our purpose. This can be obtained for extremal collections as we will now state, after making a necessary definition. 
\begin{defi}
Let $(\TT,Y)_\delta$ be a collection of $\delta$-tubes and let $\Tilde{T}$ be a $\rho$-tube that covers each tube in $\TT$. The \textit{unit rescaling} of $(\TT,Y)_\delta$ relative to $\Tilde{T}$ is the pair $(\hat{\TT},\hat{Y})_{\delta/\rho}$ defined as follows. The coaxial lines of the tubes $\hat{T}\in\hat{\TT}$ are the images of the coaxial lines of the tubes $T\in \TT$ under the map $\phi:\RR^n\to \RR^n$, defined as follows. If $\Tilde{l}$ is the coaxial line of $\Tilde{T}$, then $\phi$ is the composition of a rigid transformation that sends the point $\Tilde{l}\cap \{x_n=0\}$ to the origin and sends the line $\Tilde{l}$ to the $x_n$-axis, with the anisotropic dilation
$$
(x_1,\ldots,x_n)\mapsto (cx_1/\rho,cx_2/\rho,\ldots,cx_{n-1}/\rho,cx_n).
$$
Each shading $\hat{Y}(\hat{T})$ is the union of $\delta/\rho$ cubes that intersect $\phi(Y(T))$. The constant $c=c(n)$ is chosen so that the coaxial line of the tubes in $\hat{T}$ are in $\mathcal{L}_n$ and $\hat{Y}(\hat{T})\subset \hat{T}$. See the figure below for a visualization.
\end{defi}

\begin{figure}[h]
\includegraphics[scale=0.5]{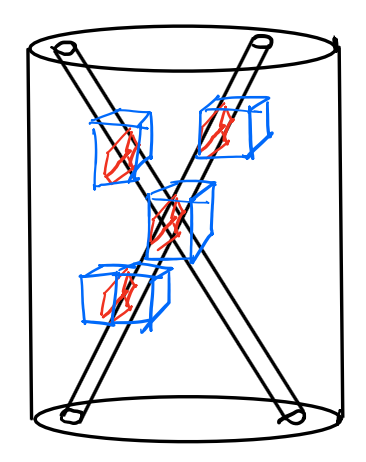}
\centering
\caption{Unit rescaling. For clarity, only a handful of cubes drawn.}
\label{rescalepic}
\end{figure}

We are now ready to state the result about the multi-scale self-similarity of extremal tubes. This is a minor variant of Prop 3.2 in Wang-Zahl \cite{wangzahl}. As we have said, we are borrowing the work done in their paper on discretizing sticky Kakeya sets. Again, we refer the interested reader to Wang-Zahl \cite{wangzahl} to find the proof of this proposition.

\begin{prop}\label{multiscaleprop}
    For all $\epsilon>0$, there exists $\eta>0$ and $\delta_0>0$ so that the following holds for all $\delta\in (0,\delta_0]$. Let $(\TT,Y)_\delta$ be an $\eta$-extremal collection of tubes, and let $\rho\in[\delta^{1-\epsilon},\delta^\epsilon]$. Then there is a refinement $(\TT',Y')_\delta$ of $(\TT,Y)_\delta$ and a cover $(\Tilde{\TT},\Tilde{Y})_\rho$ of $(\TT',Y')_\delta$ with the following properties.
    \begin{enumerate}[(i)]
         \item Each tube $T\in \TT'$ is covered by a single tube from $\Tilde{\TT}$.
        \item $(\Tilde{\TT},\Tilde{Y})_\rho$ is $\epsilon$-extremal.
        \item For each $\Tilde{T}\in \Tilde{\TT}$, the unit rescaling of $(\TT'[\Tilde{T}],Y')_\delta$ relative to $\Tilde{T}$ is $\epsilon$-extremal.
        \item For each $\Tilde{Q}\in \mathcal{Q}(\Tilde{Y})$, we have $\#\Tilde{\TT}_{\Tilde{Y}}(\Tilde{Q})\lesssim \rho^{-\sigma_n-\epsilon}$. 
        \item For each $Q\in \mathcal{Q}(Y')$ and $\Tilde{T}\in \Tilde{\TT}$, we have $\#\TT'[\Tilde{T}]_{Y'}(Q)\lesssim (\delta/\rho)^{-\sigma_n-\epsilon}$. Furthermore, there is a number $\mu_{\text{fine}}$ such that for each $\delta$-cube $Q\subset E_{\TT'[\Tilde{T}]}$, we have $\#\TT'[\Tilde{T}]_{Y'}(Q)\sim \mu_{\text{fine}}$. 
        \item There exists a constant $w>0$ such that $\sum_{T\in \TT'[\Tilde{T}]} |Y'(T)\cap \Tilde{Q}|\sim w$ for each $\Tilde{T}\in \Tilde{\TT}$ and $\Tilde{Q}\in \Tilde{Y}(\Tilde{T})$. Furthermore using (iv), we get that $|E_{\TT'[\Tilde{T}]\cap \Tilde{Q}}|\sim \frac{w}{\mu_{\text{fine}}}$. 
    \end{enumerate}
\end{prop}

\begin{remark}\label{multiscalermk}
    The Properties (ii) and (iii) allow us to apply induction on scales on extremal collections of tubes. The average multiplicity of an extremal collection of $\delta$-tubes is roughly $\delta^{-\sigma_n}$ as we saw in Remark \ref{hugermk}. Therefore, Properties (iv) and (v) say that the multiplicity at any point is roughly average, whether we are considering the coarse multiplicity $\mu_{\text{coarse}}$ at the scale $\rho$ or the multiplicity $\mu_\delta$ at scale $\delta$.  These properties also allow us to justify \eqref{multieq} which stated $\mu_\delta=\mu_{\text{fine}}\mu_{\text{coarse}}$. This regularity will be most useful to us in both Sections \ref{mainlemsec} and \ref{pfsec}. Observe that (iv) and (v) are upper bounds and these may be combined together effectively and converted into lower bounds on volumes of unions of tubes. This strategy will be used in Section \ref{pfsec} where we actually prove Theorem \ref{goal} using induction on scales. The Property (vi) tells us that the smaller $\delta$ cubes are not arbitrarily distributed within the larger $\rho$-cubes. As discussed in Remark \ref{canonrmk}, this property together with (i) tells us that canonical refinements lose the same fraction of mass as the refinement at the larger scale.
\end{remark}

\section{Trilinear and plany cases for sticky Kakeya sets}\label{mainlemsec}

From here onwards, we shall always work in the ambient space $\RR^4$ as this is the dimension we are concerned with. In this section, we shall prove a lemma which splits our sticky Kakeya set into the trilinear and plany cases, critically using the multi-scale self-similarity coming from Proposition \ref{multiscaleprop}. This lemma plays a key role when we do induction on scales in Section \ref{pfsec} where we prove Theorem \ref{goal}. However, before we get to this lemma, first we state a proposition which we will need to prove our lemma. Informally, what this proposition tells us is that a given set of unit vectors in $\RR^4$ is either trilinear or it is plany. The precise formulation of what trilinear and plany mean in this context will be presented in the statement of the lemma itself. Once we understand how to split between the trilinear and plany case for the more basic situation of  unit vectors in $\RR^4$, we can then move to the more complicated situation of dealing with tubes coming from sticky Kakeya sets. Note that this is a special subcase of Lemma 2.2 from Zahl \cite{zahl2}.

\begin{prop}\label{unilines}
    Let $U\subset S^3$  be a multiset of unit vectors and let $0\leq \rho \leq 1$. Then at least one of the following must hold:
    \begin{equation}\label{case1}
        \#\{(u_1,u_2,u_3)\in U^3: |u_1\wedge u_2 \wedge u_3|\gtrsim \rho^2\}\gtrsim (\# U)^3,
    \end{equation}
    or there is a 2-plane $H$ such that
    \begin{equation}\label{case2}
        \#\{u\in U: \angle (u, H)\lesssim \rho\} \gtrsim (\#U).
    \end{equation}
\end{prop}
In the above statement, \eqref{case1} corresponds to $U$ being a trilinear collection of unit vectors and \eqref{case2} corresponds to $U$ being a plany collection of unit vectors. While we have spoken about planyness, we have never formally defined it for a collection of tubes. Therefore, we do that now.

\begin{defi}\label{planydef}
    Let $(\Tilde{\TT},\Tilde{Y})_\rho$ be a set of $\rho$-tubes and their associated shading. We say that $(\Tilde{\TT},\Tilde{Y})_\rho$ is \textit{plany} if for each $\rho$-cube $\Tilde{Q}$, there is a plane $\Pi(\Tilde{Q})$ so that for all $\Tilde{T}\in \Tilde{\TT}_{\Tilde{Y}}(\Tilde{Q})$ we have $\angle (v(\Tilde{T}),\Pi(\Tilde{Q})\lesssim \rho$.
\end{defi}

The above definition is quite a strong one. For such collections, Katz-Zahl \cite{katzzahl} obtained the strong dimension bound of 10/3. We shall state their result as Theorem \ref{planebrush} in this paper. Planiness as defined above will be a key hypothesis in the statement. Now we are ready to state our main lemma which splits sticky Kakeya sets into trilinear and plany cases. 

\begin{lem}\label{mainlem}
    For all $\epsilon>0$ there exists constants $\eta>0$ and $\delta_0>0$ such that the following holds for all $\delta\in (0,\delta_0]$. Let $(\TT,Y)_\delta$ be an $\eta$-extremal collection of tubes and let $\rho\in [\delta^{1-\epsilon},\delta^\epsilon]$. Then one of the following must be true.
    \begin{enumerate}[(a)]
        \item There exists a refinement $(\TT',Y')_\delta$ of $(\TT,Y)_\delta$ such that for each $Q\in \mathcal{Q}(Y')$,
        $$
         \#\{(T_1,T_2,T_3)\in \TT'_{Y'}(Q)^3: |v(T_1)\wedge v(T_2)\wedge v(T_3)|\gtrsim \rho^2 \}\gtrsim (\#\TT'_{Y'}(Q))^3.
        $$
        \item There exists a $\delta^{23\eta+2\epsilon^3}$-refinement $(\TT',Y')_\delta$ of $(\TT,Y)_\delta$, a cover $(\Tilde{\TT},\Tilde{Y})_\rho$ of $(\TT',Y')_\delta$ with the following properties.
        \begin{enumerate}[(i)]
        \item Each tube $T\in \TT'$ is covered by a single tube from $\Tilde{\TT}$.
        \item $(\Tilde{\TT},\Tilde{Y})_\rho$ is $\epsilon$-extremal.
        \item For each $\Tilde{T}\in \Tilde{\TT}$, the unit rescaling of $(\TT'[\Tilde{T}],Y')_\delta$ relative to $\Tilde{T}$ is $\epsilon$-extremal.
        \item For each $\Tilde{Q}\in \mathcal{Q}(\Tilde{Y})$, we have $\#\Tilde{\TT}_{\Tilde{Y}}(\Tilde{Q})\lesssim \rho^{-\sigma_4-\epsilon}$.
        \item For each $Q\in \mathcal{Q}(Y')$ and $\Tilde{T}\in \Tilde{\TT}$, we have $\#\TT'[\Tilde{T}]_{Y'}(Q)\lesssim (\delta/\rho)^{-\sigma_n-\epsilon}$. Furthermore, there is a number $\mu_{\text{fine}}$ such that for each $\delta$-cube $Q\subset E_{\TT'[\Tilde{T}]}$, we have $\#\TT'[\Tilde{T}]_{Y'}(Q)\sim \mu_{\text{fine}}$. 
        \item $(\Tilde{\TT},\Tilde{Y})_\rho$ is plany.
        \end{enumerate}
    \end{enumerate}
\end{lem}

That completes the statement of the lemma. Before we get into the proof of the lemma, it is useful to discuss the statement a bit first. As we saw in the sketch in Subsection \ref{pfske}, the proof proceeds by splitting the collection of $\delta$-tubes into the trilinear case and the plany case. In the above statement, the trilinear case corresponds to (a) and the plany case corresponds to (b). We will use Proposition \ref{unilines} to split $(\TT,Y)_\delta$ between the two cases. As we shall shortly see in the proof, using Proposition \ref{unilines} we get that either the collection is trilinear which corresponds to (a) or that the $\delta$-tubes through any point would lie in a $\rho$-neighborhood of a plane. In the latter case, by using Proposition \ref{multiscaleprop}, we can show that the fatter tubes at scale $\rho$ are plany in terms of Definition \ref{planydef}. Even though we get a strong property on the scale $\delta$ using \eqref{case2}, to get the required plany property at scale $\rho$ takes a fair bit of technical work. This is roughly how the proof of the lemma goes.

When we actually are proving Theorem \ref{goal} in Section \ref{pfsec}, the trilinear case corresponding to (a) can be easily dealt with by applying an estimate from Guth-Zahl \cite{guthzahl} which is stated in this paper as Theorem \ref{trikak}. The plany case corresponding to (b) will be more complicated to deal with there as well. Because of the hard-earned planyness at scale $\rho$, we can apply planebrush bound i.e Theorem \ref{planebrush} from Katz-Zahl \cite{katzzahl}. However, we ultimately require a bound at scale $\delta$. For this, note that part (b) also gives us multi-scale self-similarity properties which are basically the properties in Proposition \ref{multiscaleprop}. Therefore, in this situation we will apply induction on scales along with the planebrush bound. Full detail will of course be provided when we do the actual proof in Section \ref{pfsec}. Now that we have a sketch and motivation for this lemma, let us actually prove it.

\begin{proof}[Proof of Lemma \ref{mainlem}]

The first thing we would like to do is create some regularity in the multiplicity so that we can apply Proposition \ref{unilines} more effectively. Therefore, using Lemma \ref{constaref}, we first do a constant multiplicity refinement on the $\eta$ (to be chosen later) extremal collection $(\TT,Y)_\delta$ to obtain the collection $(\TT_1,Y_1)_\delta$. Then for each $Q\in \mathcal{Q}(Y_1)$ we apply Lemma \ref{unilines} to the set of directions of the coaxial lines of the tubes in $(\TT_1)_{Y_1}(Q)$. We have the same number of unit vectors/tubes for each cube because of our constant multiplicity refinement and this will be quite useful for us as we shall now see. Suppose the majority of the cubes in $\mathcal{Q}(Y_1)$ fell into the case given by \eqref{case1}, then we drop the remaining cubes. This causes us to lose at most half the cubes, and since the cubes all have constant multiplicity, this is a valid refinement which we denote as $(\TT',Y')_\delta$. Note that $(\TT',Y')_\delta$ satisfies (a) of the lemma. Therefore, in this scenario which is the trilinear case, we are done. 

Now we need to show that (b) holds in the alternative scenario, namely the plany case. The trilinear case i.e. part (a) of the lemma took only a paragraph and two refinements to prove. However the plany case will require five pages and also five refinements to prove. There is quite a contrast in difficulty between the two cases. Let us get started with the long process now. 

In the alternative scenario the majority of the cubes fell into the case given by \eqref{case2}. Therefore, just like before, we keep these cubes only and drop the remaining cubes. Recall that as we saw in \eqref{refineineqlem2}, a shading is essentially a collection of cube-tube incidence pairs. Keeping this in mind, we refine our shading further. Let $Q$ be a cube we have kept in the previous step. Now, among the tubes $T$ incident on $Q$ (that is when $Q\subset Y_1(T)$) we only keep the incidence pairs $(Q,T)$ for which the direction $u$ of the coaxial line of $T$ satisfies $\angle (u, H_Q)\lesssim \rho$. Here $H_Q$ denotes the 2-plane associated to $Q$. Doing this ensures that all the $\delta$-tubes passing any $\delta$-cube lie in the $\rho$-neighborhood of a 2-plane. Moreover, this will be a refinement as we retain a constant fraction of mass at both steps (this is due to the fact that $(\TT_1,Y_1)$ had constant multiplicity property and secondly due to the lower bound property in \eqref{case2}). Also conveniently, applying any further refinement will automatically preserve this property. We call the refined collection we get $(\TT_2,Y_2)_\delta$. Because of the refinements we have done, $(\TT_2,Y_2)_\delta$ is ${2\eta}$-extremal by Remark \ref{smolrmk}. Note that our goal as reflected in Property (vi) of (b) is to show planyness in terms of Definition \ref{planydef} at a larger scale $\rho$, so we will need to use this property we just established to eventually get (vi). Note that as of now, there are no $\rho$-tubes in the picture at all, so let us introduce them next.

By choosing $\eta,\delta_0$ sufficiently small we apply Proposition \ref{multiscaleprop} with $\epsilon^3/4$ in place of $\epsilon$ to obtain a refinement of $(\TT_2,Y_2)_\delta$ which we will call $(\TT_3,Y_3)_\delta$ and a cover $(\Tilde{\TT}_3,\Tilde{Y}_3)_\rho$ with properties as described in the statement of the proposition. We chose the smaller number $\epsilon^3/4$ in place of $\epsilon$ to give ourselves some wiggle room for the forthcoming technical manipulations. At this stage we have (i), (ii), (iii), (iv), (v) of the lemma as they are the same as (i), (ii), (iii), (iv), (v) of Proposition \ref{multiscaleprop}. Therefore, we need to now work on establishing (vi) which is the planyness condition at scale $\rho$ while also making sure we don't lose the other properties currently in hand as we make further refinements.

We had previously shown that the $\delta$-tubes passing any $\delta$-cube are contained in the $\rho$-neighborhood of a 2-plane and this property still holds now because of its stability under refinements. However, when we move up to scale $\rho$, a $\rho$-cube could have multiple $\delta$-cubes inside of it each of which have different associated 2-planes (see Figure \ref{enemypic}). Therefore, we cannot directly conclude that $(\Tilde{\TT}_3,\Tilde{Y}_3)_\rho$ is plany. To handle this situation, we will first show there is a high multiplicity $\delta$-sub cube inside most of our $\rho$-cubes. Then we can argue that for these $\rho$-cubes, there are many $\rho$-tubes passing through it which contain the corresponding high multiplicity $\delta$-sub cube inside it. Retaining only these $\rho$-tubes finally gives us planyness at scale $\rho$.

\begin{figure}[h]
\includegraphics[scale=0.4]{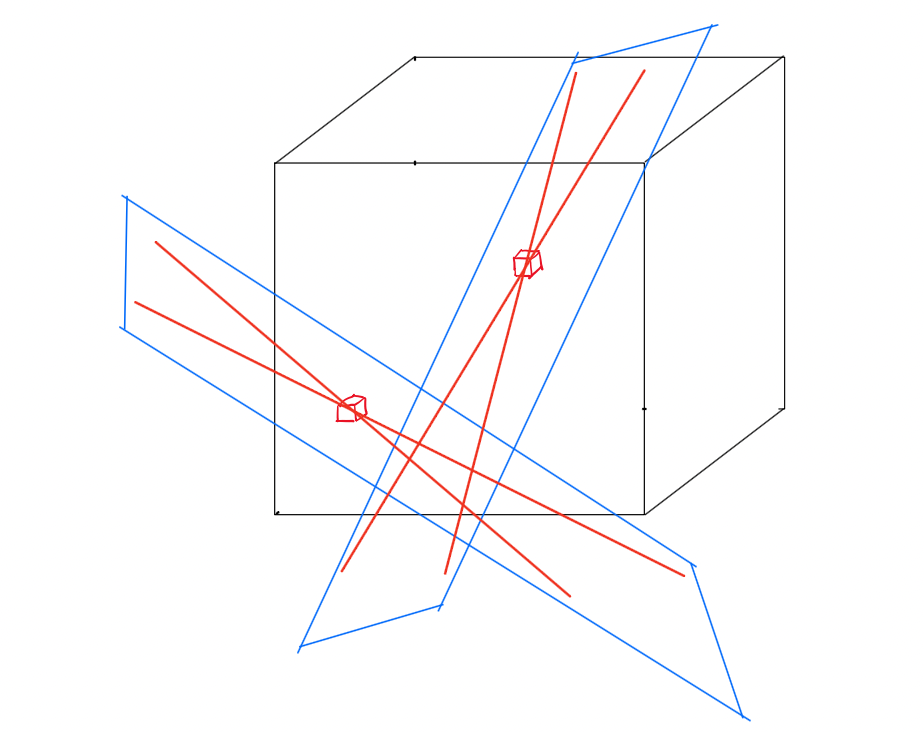}
\centering
\caption{The bad scenario of multiple 2-planes associated to different $\delta$-sub cubes.}
\label{enemypic}
\end{figure}

Let us expand on this heuristic some more using the rough multiplicity relation \eqref{multieq} which stated that $\mu_\delta=\mu_{\text{fine}}\mu_{\text{coarse}}$. Recall that this was justified by the Properties (iv) and (v) of Proposition \ref{multiscaleprop} as was discussed in Remark \ref{multiscalermk}. Now, consider a $\rho$-cube $\Tilde{Q}$ and its sub-cube $Q$ at scale $\delta$. The multiplicity at $Q$ is $\mu_\delta$ and all these $\delta$-tubes are contained in a $\rho$-neighborhood of a plane. However, at most $\mu_{\text{fine}}$ many of these tubes can be covered by a single $\rho$-tube. Therefore to cover these $\mu_\delta$ many $\delta$-tubes, we need at least $\mu_\delta/\mu_{\text{fine}}$ many $\rho$-tubes. We see that this number is simply $\mu_{\text{coarse}}$ by using the heuristic equation $\mu_\delta=\mu_{\text{fine}}\mu_{\text{coarse}}$. Furthermore, all of these $\mu_{\text{coarse}}$ many $\rho$-cubes incident on $\Tilde{Q}$ are contained in the $\rho$-neighborhood of the 2-plane corresponding to $Q$, giving us planyness. Obviously, this and the previous paragraph are only heuristics. We shall now proceed to make this intuitive argument rigorous.

Since the following arguments/calculations are very technical, we would first like to compile a series of master equations regarding $(\TT_3,Y_3)_\delta$ and $(\Tilde{\TT}_3,\Tilde{Y}_3)_\rho$ that will be used a lot in the rest of the proof of this lemma. All of these inequalities were derived in Remark \ref{hugermk} and Lemma \ref{hugermk2}, where we discussed the size, mass, density, volume, multiplicity of extremal sets. Observe that  $(\TT_3, Y_3)_\delta$ is a refinement of the $2\eta$-extremal collection $(\TT_2,Y_2)_\delta$, so by Remark \ref{smolrmk} we get that $(\TT_3, Y_3)_\delta$ is $3\eta$-extremal. Furthermore, by Property (ii) of Proposition \ref{multiscaleprop}, $(\Tilde{\TT}_3,\Tilde{Y}_3)_\rho$ is $\epsilon^3/4$-extremal. Now by applying Remark \ref{hugermk} and Lemma \ref{hugermk2}, we get the following inequalities.
\begin{equation}\label{A1}
    \delta^{-3+3\eta}\lesssim \#\TT_3 \lesssim \delta^{-3-3\eta}.
\end{equation}
\begin{equation}\label{A2}
    \delta^{3\eta}\lesssim\sum_{T\in \TT_3} |Y_3(T)|\lesssim \delta^{-3\eta}. 
\end{equation}
\begin{equation}\label{A3}
    \lambda_{Y_3}\gtrsim \delta^{6\eta}.
\end{equation}
\begin{equation}\label{A4}
    \delta^{\sigma_4+\epsilon^3/4}\lesssim \bigg|\bigcup_{T\in \TT} Y_3(T)\bigg|\lesssim \delta^{\sigma_4-3\eta}.
\end{equation}
\begin{equation}\label{A5}
    \delta^{-\sigma_4+6\eta}\lesssim\mu_{Y_3}\lesssim \delta^{-\sigma_4-\epsilon^3/4}.
\end{equation}

Similarly for $(\Tilde{\TT}_3,\Tilde{Y}_3)_\rho$, we get
\begin{equation}\label{B1}
    \rho^{-3+\epsilon^3/4}\lesssim \#\Tilde{\TT}_3 \lesssim \rho^{-3-\epsilon^3/4}.
\end{equation}
\begin{equation}\label{B2}
    \rho^{\epsilon^3/4}\lesssim\sum_{\Tilde{T}\in \Tilde{\TT}_3} |\Tilde{Y}_3(\Tilde{T})|\lesssim \rho^{-\epsilon^3/4}. 
\end{equation}
\begin{equation}\label{B3}
    \lambda_{\Tilde{Y}_3}\gtrsim \rho^{\epsilon^3/2}.
\end{equation}
\begin{equation}\label{B4}
    \bigg|\bigcup_{\Tilde{T}\in \Tilde{\TT}_3} \Tilde{Y}_3(\Tilde{T})\bigg|\lesssim \rho^{\sigma_4-\epsilon^3/4}.
\end{equation}
\begin{equation}\label{B5}
    \rho^{-\sigma_4+2\epsilon^3}\lesssim\mu_{\Tilde{Y}_3}\lesssim \rho^{-\sigma_4-\epsilon^3/4}.
\end{equation}
In the last equation i.e \eqref{B5}, we actually used Property (iv) of Proposition \ref{multiscaleprop} to get the right hand side but everything else follows directly from the inequalities in Remark \ref{hugermk} and Lemma \ref{hugermk2}. We omitted the left hand side of the penultimate equation i.e. \eqref{B4} as it would involve introducing a new parameter larger than $\epsilon$ and we don't require this side of the inequality for our purposes.

Equipped with all of these inequalities now, let us now move forward towards our objectives. First combining the Properties (iv) and (v) from Proposition \ref{multiscaleprop} we get that for any $Q\in \mathcal{Q}(Y_3)$ we have 
\begin{equation}\label{multieq2}
    \#(\TT_3)_{Y_3}(Q)\lesssim (\delta/\rho)^{-\sigma_4-\epsilon^3/4} \rho^{-\sigma_4-\epsilon^3/4} \lesssim \delta^{-\sigma_4-\epsilon^3/4}.
\end{equation}
Now combining this with the left hand side of \eqref{A5}, we get that  
\begin{equation}
    \#(\TT_3)_{Y_3}(Q) \lesssim \delta^{-6\eta-\epsilon^3/4} \mu_{Y_3}. 
\end{equation}
The above inequality says that the multiplicity at any point at scale $\delta$ is bounded by the average multiplicity multiplied by a negative subpolynomial factor. Let $S_1$ be the set of cubes $Q\in \mathcal{Q}(Y_3)$ such that $\#(\TT_3)_{Y_3}(Q)\geq \mu_{Y_3}/4$.

\underline{Claim}: We have $\# S_1 \gtrsim \delta^{6\eta+\epsilon^3/4}\# \mathcal{Q}(Y_3)$. 

\underline{Proof of claim}: We just proved that $\#(\TT_3)_{Y_3}(Q) \lesssim \delta^{-6\eta-\epsilon^3/4} \mu_{Y_3}$. By definition, 
$$
\sum_{Q\in \mathcal{Q}(Y_3)} \#(\TT_3)_{Y_3}(Q)=\mu_{Y_3}\# \mathcal{Q}(Y_3).
$$
Now clearly $\sum_{Q\notin S_1} \#(\TT_3)_{Y_3}(Q)\leq (\mu_{Y_3}\# \mathcal{Q}(Y_3))/4$. Hence we get
$$
\#S_1 \delta^{-6\eta-\epsilon^3/4} \mu_{Y_3}\gtrsim \sum_{Q\in S_1} \#(\TT_3)_{Y_3}(Q)\geq (\mu_{Y_3}\mathcal{Q}(Y_3))/2.
$$
Rearranging the above inequality we get that $\#S_1\gtrsim  \delta^{6\eta+\epsilon^3/4}\# \mathcal{Q}(Y_3)$. Thus our claim is established.

Now we want to show that this set is spread widely across the $\rho$-cubes. For this, we first try to get an upper bound on the number of cubes from $S_1$ that lie inside a single $\rho$-cube. Before moving forward, it is useful to keep in mind that the average multiplicities at scale $\delta$ and $\rho$ are roughly $\delta^{-\sigma_4}$ and $\rho^{-\sigma_4}$ as reflected in \eqref{A5} and \eqref{B5}. Similarly, the total number of $\delta$ and $\rho$-cubes are roughly $\delta^{\sigma_4-4}$ and $\rho^{\sigma_4-4}$ respectively which follows from \eqref{A4} and \eqref{B4} respectively. These factors will keep appearing in the following calculations. Using the Properties (iv) and (vi) of the Proposition \ref{multiscaleprop}, we have that in any $\rho$-cube $\Tilde{Q}\in \mathcal{Q}(\Tilde{Y_3})$
\begin{equation}\label{A11}
    \sum_{T\in \TT_3} |Y_3(T)\cap \Tilde{Q}|\lesssim w\rho^{-\sigma_4-\epsilon^3/4}.
\end{equation}
Here $w$ is the constant we get from (vi) of Proposition \ref{multiscaleprop}. More specifically, we have that $\sum_{T\in \TT_3[\Tilde{T}]} |Y_3(T)\cap \Tilde{Q}|\sim w$ for each $\Tilde{T}\in \Tilde{\TT}_3$ and $\Tilde{Q}\in \Tilde{Y}_3(\Tilde{T})$. Now as discussed in Remark \ref{canonrmk}, we see that this can be interpreted as saying the ``$\delta$-mass" inside every incidence pair at the larger scale $\rho$ is constant and that constant is $w$. Therefore, to compute $w$, we just need to divide the total $\delta$-mass which is $\sum_{T\in \TT_3}|Y_3(T)|$ by the number of incidence pairs $(\Tilde{Q},\Tilde{T})$ at the larger scale $\rho$. Using \eqref{refineineqlem2}, we can see that the number of such pairs equals $\rho^{-4}\sum_{\Tilde{T}\in \Tilde{\TT}_3}|\Tilde{Y}_3(\Tilde{T})|$. Therefore we get that
\begin{equation}\label{A12}
    w\sim \rho^4 \frac{\sum_{T\in \TT_3}|Y_3(T)|}{\sum_{\Tilde{T}\in \Tilde{\TT}_3}|\Tilde{Y}_3(\Tilde{T})|}.
\end{equation}
Combining \eqref{A11} and \eqref{A12}, we get that
\begin{equation}
    \sum_{T\in \TT_3} |Y_3(T)\cap \Tilde{Q}|\lesssim \rho^4 \frac{\sum_{T\in \TT_3}|Y_3(T)|}{\sum_{\Tilde{T}\in \Tilde{\TT}_3}|\Tilde{Y}_3(\Tilde{T})|}\rho^{-\sigma_4-\epsilon^3/4}\lesssim \rho^4 \delta^{-3\eta}\rho^{-\epsilon^3/4}\rho^{-\sigma_4-\epsilon^3/4}.
\end{equation}
At the last step we used \eqref{A2} and \eqref{B2} to bound the total masses. Note that using the same logic as in \eqref{refineineqlem2}, we get that the left hand side equals $\delta^4\sum_{Q\subset \Tilde{Q}} \# (\TT_3)_{Y_3}(Q)$. Therefore, ultimately, we get that for any $\rho$-cube $\Tilde{Q}\in \mathcal{Q}(\Tilde{Y_3})$, we have
\begin{equation}\label{eq10}
    \sum_{Q\subset \Tilde{Q}} \# (\TT_3)_{Y_3}(Q) \lesssim \delta^{-4-3\eta} \rho^{4-\sigma_4-\epsilon^3/2}
\end{equation}

This inequality can now be leveraged to show that not too many cubes from $S_1$ can be packed into a single $\rho$-cube. Using the definition of $S_1$ and \eqref{eq10}, we get that the number of $\delta$-cubes of $S_1$ inside a $\rho$-cube $\Tilde{Q}$ is $\lesssim \delta^{-4-3\eta} \rho^{4-\sigma_4-\epsilon^3/2}\mu_{Y_3}^{-1}$. Let $\Tilde{S}_1$ be the set of $\rho$-cubes $\Tilde{Q}\in \mathcal{Q}(\Tilde{Y_3})$ such that it has a $\delta$-cube from $S_1$ inside it. Therefore, the $\delta$-cubes of $S_1$ are distributed among the $\rho$-cubes of $\Tilde{S}_1$ without too many allowed to concentrate into a single $\rho$-cube. Furthermore, since $\# S_1 \gtrsim \delta^{6\eta+\epsilon^3/4}\# \mathcal{Q}(Y_3)$, we get that
\begin{equation}
    \# \Tilde{S}_1 \gtrsim \frac{\delta^{6\eta+\epsilon^3/4}\# \mathcal{Q}(Y_3)}{\delta^{-4-3\eta} \rho^{4-\sigma_4-\epsilon^3/2}\mu_{Y_3}^{-1}} \gtrsim \frac{\delta^{6\eta+\epsilon^3/4} \delta^{\sigma_4-4+\epsilon^3/4} }{\delta^{-4-3\eta} \rho^{4-\sigma_4-\epsilon^3/2}\delta^{\sigma_4-6\eta}}\gtrsim \rho^{\sigma_4-4+\epsilon^3/2}\delta^{15\eta+\epsilon^3/2}.
\end{equation}
In the above equation, we used \eqref{A4} and \eqref{A5} to bound $ \mathcal{Q}(Y_3)$ and $\mu_{Y_3}$ respectively. Also note that $\#\mathcal{Q}(\Tilde{Y}_3)\lesssim \rho^{\sigma_4-4-\epsilon^3/4}$ due to \eqref{B4}. Therefore, we get
\begin{equation}\label{main1}
    \# \Tilde{S}_1 \gtrsim \#\mathcal{Q}(\Tilde{Y}_3) \rho^{\epsilon^3} \delta^{15\eta+\epsilon^3/2}.
\end{equation}
The above inequality tells us that $\Tilde{S}_1$ consists of most of the $\rho$-cubes up to some subpolynomial losses. The next step as we mentioned earlier is to show that each of these $\rho$-cubes have a lot of $\rho$-tubes making angle $\rho$ with a 2-plane. Recall that for every $\delta$-cube $Q\in \mathcal{Q}(Y_3)$, there exists a 2-plane $H_Q$ such that all the $\delta$-tubes in $(\TT_3)_{Y_3}(Q)$ make angle less than $\rho$ with it. Therefore, let $\Tilde{Q}\in \Tilde{S}_1 $ with $\delta$ cube $Q\in S_1$ contained inside $\Tilde{Q}$. Then, using the upper bound for multiplicity of $\delta$-tubes within a $\rho$-tube from Property (iv) of Proposition \ref{multiscaleprop}, we get that the number of $\rho$-tubes $\Tilde{T}\in (\Tilde{\TT}_3)_{\Tilde{Y}_3}(\Tilde{Q})$ such that $Q\subset E_{\TT_3[\Tilde{T}]}$ is
\begin{equation}\label{main2}
    \gtrsim \frac{\mu_{Y_3}}{(\delta/\rho)^{-\sigma_4-\epsilon^3/4}}\gtrsim \frac{\delta^{-\sigma_4+6\eta}}{(\delta/\rho)^{-\sigma_4-\epsilon^3/4}}=\delta^{6\eta+\epsilon^3/4}\rho^{-\sigma_4-\epsilon^3/4}\gtrsim \delta^{6\eta+\epsilon^3/4} \mu_{\Tilde{Y}_3}.
\end{equation}
Note that we also used inequalities \eqref{A5} and \eqref{B5} to get the necessary bounds on $\mu_{Y_3}$ and $\mu_{\Tilde{Y}_3}$ respectively. All the $\rho$-tubes in the above set make angle $\lesssim \rho$ with the 2-plane $H_Q$ corresponding to $Q$. Therefore, keeping only such $\rho$-tubes among the already selected $\rho$-cubes, we obtain a $t$-refinement $(\Tilde{\TT}_4,\Tilde{Y}_4)_\rho$ of $(\Tilde{\TT}_3,\Tilde{Y}_3)_\rho$ which is plany (value of $t$ to be determined next). Furthermore, at the smaller scale, we refine $(\TT_3,Y_3)_\delta$ to $(\TT_4,Y_4)_\delta$ canonically as in Definition \ref{canondef} to maintain the covering property. 

We have now secured Lemma \ref{mainlem}, Item (b.vi), but because of the refinements we have done, we need to investigate and fix whatever damage we have done to the other properties namely (i), (ii), (iii), (iv) and (v). The good news is that (vi) is stable under further refinements. The same holds for (i) obviously. Also (iv) and (v) are upper bounds on multiplicities so they too are stable under refinements. Therefore, we mainly need to worry about (ii) and (iii) now. Let us first ensure we have not lost too much mass in our last refinement. Combining \eqref{main1} and \eqref{main2}, we get that
\begin{align}\label{eq14}
\begin{split}
    \sum_{\Tilde{T}\in \Tilde{\TT}_4} |\Tilde{Y}_4(\Tilde{T})|&=\rho^4\sum_{\Tilde{Q}\in \Tilde{S}_1}\# (\Tilde{\TT}_4)_{\Tilde{Y}_4}(\Tilde{Q})\\
    &\gtrsim \rho^4 \delta^{6\eta+\epsilon^3/4} \mu_{\Tilde{Y}_3} \#\mathcal{Q}(\Tilde{Y}_3) \rho^{\epsilon^3} \delta^{15\eta+\epsilon^3/2}\\
    &\gtrsim \rho^{\epsilon^3}\delta^{21\eta+\epsilon^3}\sum_{\Tilde{T}\in \Tilde{\TT}_3} |\Tilde{Y}_3(\Tilde{T})|.
\end{split}
\end{align}
Note that in the above, we freely used \eqref{refineineqlem2} to switch between sum over multiplicities and masses of shadings. Now because of Proposition \ref{multiscaleprop} Items (i) and (vi)  (not to be confused
with Lemma \ref{mainlem}, Items (b.vi)), we have the same refinement factor at both scales $\rho$ and $\delta$. This was shown in Remark \ref{canonrmk}. Hence we get
\begin{equation}\label{eq15}
    \sum_{T\in \TT_4} |Y_4(T)|\gtrsim \rho^{\epsilon^3}\delta^{21\eta+\epsilon^3}\sum_{T\in \TT_3}|Y_3(T)|.
\end{equation}

Now some of the fat tubes could have low mass inside of them violating (ii), so we fix that now. Dyadically pigeonhole to get constant density at the scale $\rho$ by throwing away some fat tubes. Call the obtained collection $(\Tilde{\TT}_5,\Tilde{Y}_5)_\rho$ and again refine $(\TT_4,Y_4)_\delta$ to $(\TT_5,Y_5)_\delta$ canonically to maintain the covering property. Since we have the same refinement factor at both scales (again from (vi) of Proposition \ref{multiscaleprop}), the collections $(\TT_5,Y_5)_\delta$ and $(\Tilde{\TT}_5,\Tilde{Y}_5)_\rho$ are refinements of $(\TT_4,Y_4)_\delta$ and $(\Tilde{\TT}_4,\Tilde{Y}_4)_\rho$ respectively. Therefore 
\begin{equation*}
    \sum_{T\in \TT_5[\Tilde{T}]} |Y_5(T)|\sim \frac{\sum_{T\in \TT_5}|Y_5(T)|}{\# \Tilde{\TT}_5}\gtrsim \frac{\rho^{\epsilon^3}\delta^{22\eta+\epsilon^3} \sum_{T\in \TT}|Y(T)|}{\# \Tilde{\TT}_3}\gtrsim \frac{\rho^{\epsilon^3}\delta^{23\eta+\epsilon^3}}{\rho^{1-n-\epsilon^3/4}}.
\end{equation*}
Then after unit rescaling, we get
$$
\sum_{\hat{T}\in \hat{\TT}} |\hat{Y}(\hat{T})|\gtrsim \rho^{2\epsilon^3}\delta^{23\eta+\epsilon^3}\gtrsim \delta^{23\eta+3\epsilon^3}\gtrsim \bigg(\frac{\delta}{\rho}\bigg)^{\frac{23\eta+3\epsilon^3}{\epsilon}}.
$$
The last inequality follows from the fact that $\rho\geq \delta^{1-\epsilon}$. Hence by choosing $\eta$ sufficiently small we can obtain $\sum_{\hat{T}\in \hat{\TT}}|\hat{Y}(\hat{T})|\geq (\delta/\rho)^\epsilon$ which establishes (iii). Now we just need to make sure (ii) holds. This is not too hard to ensure as we have been keeping track of how much mass we have been losing at scale $\rho$ and its just a few logarithmic or subpolynomial factors. Therefore, using \eqref{eq14}, we get that
$$
\sum_{\Tilde{T}\in \Tilde{\TT}_5}|\Tilde{Y}_5(\Tilde{T})| \gtrsim \rho^{2\epsilon^3}\delta^{21\eta+\epsilon^3} \gtrsim \rho^{2\epsilon^3} \rho^{\frac{21\eta+\epsilon^3}{\epsilon}}.
$$
Here we used that $\rho\leq \delta^\epsilon$ in the last inequality. By choosing $\eta$ sufficiently small, we can ensure that $\sum_{\Tilde{T}\in \Tilde{\TT}_5}|\Tilde{Y}_5(\Tilde{T})|\geq \rho^\epsilon$ which establishes (ii). Therefore $(\TT_5,Y_5)_\delta$ and $(\Tilde{\TT}_5,\Tilde{Y}_5)_\rho$ satisfy Properties (i) to (vi) of the Lemma. Now we just need to compute the refinement factor of $(\TT_5,Y_5)_\delta$. Using \eqref{eq15} and \eqref{B2}, we have
$$
\sum_{T\in \TT_5} |Y_5(T)|\gtrsim \rho^{\epsilon^3}\delta^{23\eta+\epsilon^3}\sum_{T\in \TT}|Y(T)|\gtrsim \delta^{2\epsilon^3+23\eta}\sum_{T\in \TT}|Y(T)|.
$$
Therefore, we get that $(\TT_5,Y_5)_\delta$ is a $\delta^{23\eta+2\epsilon^3}$-refinement  of $(\TT,Y)_\delta$. At last all the required properties are now established and the proof is complete.

\end{proof}

\begin{remark}\label{shrink}
    We can replace $\eta$ with anything smaller and $\epsilon$ with anything bigger and the result will still hold. This flexibility will be useful for us later on in the proof of Theorem \ref{goal} in Section \ref{pfsec}.
\end{remark}

\section{Previous Kakeya-type estimates in $\RR^4$}

\subsection{Technical preliminaries}

Before we state the pre-existing Kakeya-type estimates in $\RR^4$, we need to discuss some technicalities. First, let us make a couple of definitions. The definitions in this section are from Section 2 of \cite{katzzahl}.

\begin{defi} (Two-ends condition).
    Let $(\TT,Y)_\delta$ be a set of $\delta$-tubes and their associated shading. We say that $(\TT,Y)_\delta$ is $(\epsilon_1,C_1)$-\textit{two ends} if for all $T\in \TT$ and all $\delta\leq r\leq 1$, we have
    \begin{equation}
        \#\{Q:Q\subset Y(T)\cap B(x,r)\}\leq r^{\epsilon_1}C_1 \lambda_Y\delta^{-1}
    \end{equation}
    for all balls $B(x,r)$.
\end{defi}

Intuitively, this condition says that the shading is not concentrated in any $r$-ball but spread across both ends of the tube. The two-ends condition is valuable because collections of tubes which satisfy this condition tend to be easier to manipulate. Next, we make another definition.

\begin{defi} (Robust transversality).
    Let $(\TT,Y)_\delta$ be a set of $\delta$-tubes and their associated shading. We say that $(\TT,Y)_\delta$ is $(\epsilon_2,C_2)$-\textit{robustly transverse} if for all $\delta$-cubes $Q$, all vectors $v$, and all for all $\delta<r<1$, we have 
    \begin{equation}
        \#\{T\in \TT_Y(Q): \angle (v(T),v)\leq r\}\leq r^{\epsilon_2}C_2 \mu_Y.
    \end{equation}
\end{defi}

This property blocks too many tubes passing through a point from concentrating around a particular direction. In other words, the tubes passing through a point have transverse directions. Both of these conditions are used frequently in the research literature. The reason we made these definitions here is because both of these conditions occur as hypothesis in the planebrush bound (Theorem \ref{planebrush}) of Katz-Zahl \cite{katzzahl}.

We would like to apply this planebrush result to extremal collections of tubes. Hence, to do this, we need to show that extremal collections satisfy both two-ends and robust-transversality conditions. Property (c) of Definition \ref{31} will help us obtain the two-ends without much difficulty. However, obtaining the robust-transversality condition is a bit more challenging. We will need to use Proposition \ref{multiscaleprop} for this. Wang and Zahl \cite{wangzahl} used this very proposition to show robust-transversality at angle $\delta^\epsilon$ for extremal collections at scale $\delta$ (Lemma 3.4 in \cite{wangzahl}). We adapt their approach to obtain robust-transversality for all angles $\delta\leq r\leq 1$. 

\begin{lem}\label{tertlem}
    For all $\epsilon>0$, there exists $\eta>0$ and $\delta_0>0$ with the following property. Suppose that $(\TT,Y)_\delta$ is an $\eta$-extremal collection of tubes with $\delta<\delta_0$. Then there is a $\delta^\eta$-refinement $(\TT',Y')_\delta$ of $(\TT,Y)_\delta$ which is $(1, \delta^{-4\eta})$-two ends and $(\sigma_4, \delta^{-3\eta-3\epsilon})$-robustly transverse.
\end{lem}

\begin{proof}
Let $N=1/\epsilon$. Choose $\eta$ small enough so that $N\eta$ works for $\epsilon$ in Proposition \ref{multiscaleprop}. Let $\rho_1=\delta^\epsilon=\delta^{1/N}$. Apply Proposition \ref{multiscaleprop} with this as value of $\rho$ to get a refinement $(\TT_1,Y_1)_\delta$ and a cover $(\Tilde{\TT},\Tilde{Y})_{\rho_1}$. Since the resulting collection $(\Tilde{\TT},\Tilde{Y})_{\rho_1}$ of tubes are essentially distinct, the collection $\{T\in (\TT_1)_{Y_1}(Q):\angle(v(T),v)\leq \rho_1\}$ must be covered by $O(1)$ $\rho_1$-tubes from $\Tilde{\TT}$. Now using Item (v) of Proposition \ref{multiscaleprop} we get 
\begin{equation*}
    \# \{T\in (\TT_1)_{Y_1}(Q):\angle(v(T),v)\leq \rho_1\}\lesssim (\delta/\rho_1)^{-\sigma_4-\epsilon}.
\end{equation*}
The refined collection $(\TT_1,Y_1)_\delta$ is $2\eta$-extremal, so we can apply Proposition \ref{multiscaleprop} again to get a refinement $(\TT_2,Y_2)_\delta$ with $\rho_2=\delta^{2/N}$. Again, we get
$$
\# \{T\in (\TT_2)_{Y_2}(Q):\angle(v(T),v)\leq \rho_2\}\lesssim (\delta/\rho_2)^{-\sigma_4-\epsilon}.
$$
We keep doing this until we finally get a refinement $(\TT_{N-1},Y_{N-1})_\delta$ with $\rho_{N-1}=\delta^{\frac{N-1}{N}}=\delta^{1-\epsilon}$. Just like before, we eventually obtain  
$$
\# \{T\in {(\TT_{N-1}})_{Y_{N-1}}(Q):\angle(v(T),v)\leq \rho_{N-1}\}\lesssim (\delta/\rho_{N-1})^{-\sigma_4-\epsilon}.
$$
Let us call this final collection $(\TT',Y')_\delta$. Then as upper bounds are preserved under refinements we get
$$
\# \{T\in \TT'_{Y'}(Q):\angle(v(T),v)\leq \delta^{j/N}\}\lesssim \bigg(\frac{\delta}{\delta^{\frac{j}{N}}}\bigg)^{-\sigma_4-\epsilon}\quad \quad \text{for }j=1,\ldots,N-1.
$$
We have thus obtained an upper bound for a sequence of angles. Next, suppose that $\delta^{\frac{j+1}{N}}<\rho\leq\delta^{\frac{j}{N}}$ for $j=1,\ldots, N-1$. Then
\begin{align*}
    \# \{ T \in \TT'_{Y'}(Q): \angle (v(T),v)\leq \rho \} &\lesssim \bigg( \frac{\delta}{\delta^{\frac{j}{N}}}\bigg)^{-\sigma_4-\epsilon}\\
    &\lesssim \bigg( \frac{\delta}{\rho\delta^{\frac{-1}{N}}}\bigg)^{-\sigma_4-\epsilon}\\
    &\lesssim \bigg( \frac{\delta}{\rho}\bigg)^{-\sigma_4-\epsilon} (\delta^\epsilon)^{-\sigma_4-\epsilon}\\
    &\lesssim \bigg( \frac{\delta}{\rho}\bigg)^{-\sigma_4-\epsilon} \delta^{-2\epsilon}.
\end{align*}
Therefore, we have now obtained an upper bound for angles $\rho\in (\delta,\delta^{1/N}]$. We need to consider $\delta^{\frac{1}{N}}<\rho<1$ separately. In this case combining (iii) and (iv) of Proposition \ref{multiscaleprop}, we get
$$
\# \{ T\in \TT'_{Y'}(Q):\angle (v(T),v)\leq \rho \}\lesssim \delta^{-\sigma_4-\epsilon}\lesssim \bigg( \frac{\delta}{\rho}\bigg)^{-\sigma_4-\epsilon} \delta^{-2\epsilon}.
$$
We used similar manipulations here as in the previous computation. Therefore, we finally have
\begin{equation}\label{18}
    \# \{ T\in \TT'_{Y'}(Q):\angle (v(T),v)\leq \rho \}\lesssim \bigg( \frac{\delta}{\rho}\bigg)^{-\sigma_4-\epsilon} \delta^{-2\epsilon}\quad \text{for all } \rho\in(\delta,1).
\end{equation}
Because we did $N=1/\epsilon$ refinements, we get 
$$
\sum_{T\in \TT'}|Y'(T)|\gtrsim \bigg(\log \frac{1}{\delta}\bigg)^{-C/\epsilon}\sum_{T\in \TT}|Y(T)|.
$$
By choosing $\delta_0>0$ sufficiently small we can guarantee that
$$
\bigg(\log \frac{1}{\delta}\bigg)^{-C/\epsilon}\gtrsim \delta^\eta \quad \text{for all }\delta\in (0,\delta_0].
$$
Therefore, $(\TT',Y')_\delta$ is a $\delta^\eta$ refinement. Further, we have the following lower bound on its average multiplicity
\begin{equation}\label{19}
    \mu_{Y'}=\frac{\sum_{T\in \TT'}|Y'(T)|}{\big|\bigcup_{T\in \TT'}Y'(T)\big|}\gtrsim \frac{\delta^{2\eta}}{\delta^{\sigma_4-\eta}}=\delta^{-\sigma_4+3\eta}.
\end{equation}
Combining \eqref{18} and \eqref{19}, we get that
$$
\# \{ T\in \TT'_{Y'}(Q):\angle (v(T),v)\leq \rho \}\lesssim \rho^{\sigma_4} \delta^{-3\epsilon-3\eta}\mu_{Y'}.
$$
Hence we have established that $(\TT',Y')_\delta$ is $(\sigma_4, \delta^{-3\eta-3\epsilon})$-robustly transverse. Further note that $(\TT',Y')_\delta$ is $2\eta$-extremal, therefore $\lambda_{Y'}\gtrsim \delta^{4\eta}$ by \eqref{densityext}. This finally gives us that
$$
\# \{Q:Q\subset Y'(T)\cap B(x,r)\}\lesssim r\delta^{-1}\lesssim r \delta^{-4\eta} \lambda_{Y'} \delta^{-1}.
$$
Hence we get that $(\TT',Y')_\delta$ is $(1,\delta^{-4\eta})$-two ends. This completes the proof.

\end{proof}

\subsection{Guth-Zahl's Trilinear Kakeya estimate}

Trilinear Kakeya gives a dimension bound of $3.25$ for collections of tubes in which most triples of tubes passing through a common point have direction vectors pointing in (quantitatively) linearly independent directions (see Figure \ref{trifig}). An estimate of this kind was first proved by Guth and Zahl in \cite{guthzahl}. The following is the version which is stated in \cite{katzzahl} as Corollary 3.6.
\begin{theo}\label{trikak}
    For each $\epsilon>0$, there exists a constant $c_\epsilon$ so that the following holds. Let $(\TT,Y)_\delta$ be a set of $\delta$-tubes in $\RR^4$ that point in $\delta$-separated directions and their associated shading. Suppose that for each $Q\in \mathcal{Q}(Y)$, we have
    $$
    \#\{(T_1,T_2,T_3)\in \TT_Y(Q)^3: |v_1\wedge v_2 \wedge v_3|\geq \theta\}\geq s (\#\TT_Y(Q))^3.
    $$
    Then
    \begin{equation}
        \bigg| \bigcup_{T\in \TT} Y(T)\bigg|\geq c_\epsilon \delta^\epsilon s^{9/4} \lambda_Y^{13/4}\delta^{3/4}\theta (\delta^3\#\TT)^{1/4}.
    \end{equation}
\end{theo}

\subsection{Katz-Zahl's Planebrush estimate}

As we have mentioned earlier, the planebrush is a geometrical object which is a higher dimensional analogue of Wolff's hairbrush \cite{wolff1}. Both of these objects are represented pictorially below (Figures \ref{hbfig} and \ref{pbfig}).

\begin{figure}[h]
    \centering
    \begin{minipage}{0.45\textwidth}
        \centering
        \includegraphics[scale=0.4]{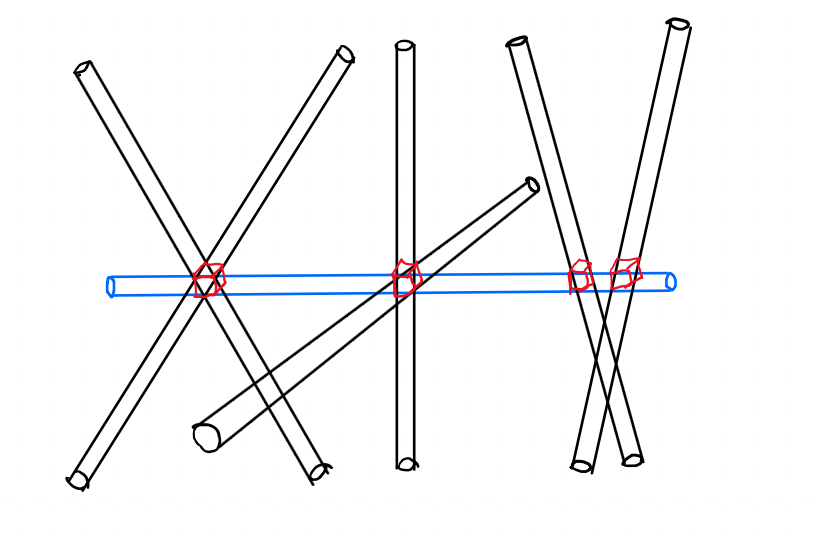} % first figure itself
        \caption{Hairbrush configuration.}
        \label{hbfig}
    \end{minipage}\hfill
    \begin{minipage}{0.45\textwidth}
        \centering
        \includegraphics[scale=0.4]{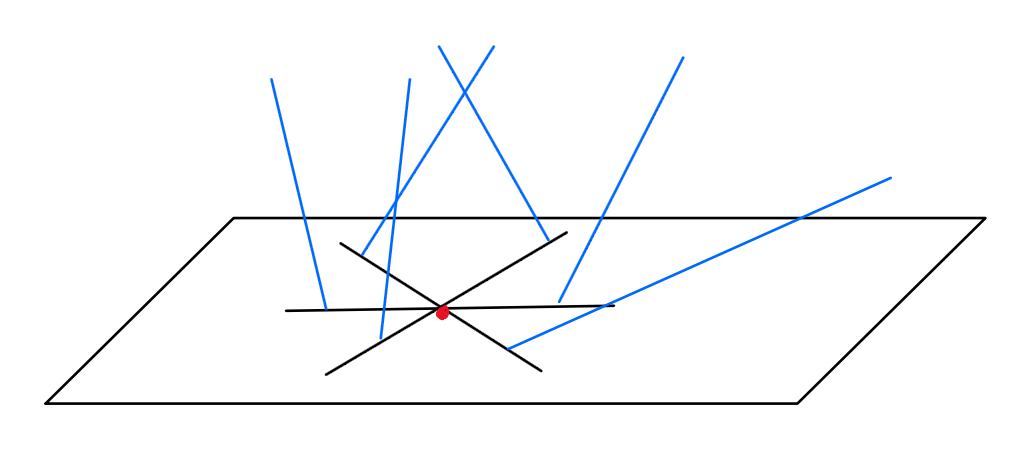} % second figure itself
        \caption{Planebrush configuration. We use lines instead of tubes for simplicity.}
        \label{pbfig}
    \end{minipage}
\end{figure} 

The planebrush argument gives a bound of $10/3$ for plany collections of tubes in $\RR^4$ as we shall state precisely now. Note that this is due to Katz and Zahl \cite{katzzahl}, and occurs as Proposition 5.3 in their paper.
\begin{theo}\label{planebrush}
    Let $0<\rho<1$. Let $0<\epsilon_2<\epsilon_1<1$. Let $\Omega$ be a set of $\rho$-separated directions. Let $(\Tilde{\TT},\Tilde{Y})_\rho$ be a set of essentially distinct plany $\rho$-tubes and their associated shading. Suppose that
    \begin{itemize}
        \item There are $\sim (\#\Tilde{\TT})/(\#\Omega)$ tubes from $\Tilde{\TT}$ pointing in each direction $v\in \Omega$.
        \item $(\Tilde{\TT},\Tilde{Y})_\rho$ is $(\epsilon_1,C_1)$-two ends.
        \item $(\Tilde{\TT},\Tilde{Y})_\rho$ is $(\epsilon_2,C_2)$-robustly transverse. 
    \end{itemize}
    Then for each $\epsilon>0$, there is a constant $c>0$ (depending on $\epsilon, \epsilon_1,\epsilon_2)$ so that 
    \begin{equation}
        \bigg| \bigcup_{\Tilde{T}\in \Tilde{\TT}} \Tilde{Y}(\Tilde{T})\bigg| \geq cC_1^{-1/\epsilon_1} C_2^{-2/\epsilon_2}\rho^\epsilon \lambda_{\Tilde{Y}}^{4/3} \rho^{2/3} (\rho^3 \# \Omega)^{1/3} (\rho^3\#\Tilde{\TT})^{2/3}.
    \end{equation}
\end{theo}

We now have all the ingredients we need to prove Theorem \ref{goal}, so we do this in the following section.

\section{Proof of Theorem \ref{goal}}\label{pfsec}

For the reader's convenience, we recall Theorem \ref{goal} here.

\begin{flushleft}
    \textbf{Theorem \ref{goal}.}  Let $n=4$. Then, for every $\epsilon>0$, there exists a constant $c_\epsilon>0$ and $\eta>0$ such that $V(\eta, \delta)\geq c_\epsilon \delta^{\frac{3}{4}+\epsilon}$. That is to say whenever $(\TT,Y)_\delta$ is an $\eta$-extremal collection of tubes in $\RR^4$, we have
    \begin{equation}\label{mainineq}
        \bigg|\bigcup_{T\in \TT} Y(T)\bigg|\geq c_{\epsilon}\delta^{\frac{3}{4}+\epsilon}.
    \end{equation}

\end{flushleft}

We shall be proving the above statement in this section. As we saw in Section \ref{dmss}, the above theorem implies Theorem \ref{main}. We recall it as well.

\begin{flushleft}
    \textbf{Theorem \ref{main}.}  Sticky Kakeya sets in $\RR^4$ have Hausdorff dimension at least 3.25.
\end{flushleft}

Thus, all that remains to be done is prove Theorem \ref{goal} using all the ingredients from previous sections. The reader may wish to take a moment to revisit the sketch we made in Subsection \ref{pfske} before diving into the following rigorous proof.

\vspace{4mm}

\textit{Proof of Theorem \ref{goal}.} We will use induction on scales to prove the theorem. Let us fix $\epsilon_0>0$, we now need to find $\eta^*>0$ so that 
\begin{equation}\label{maineq}
    V(\eta^*,\delta)\geq c_{\epsilon_0} \delta^{\frac{3}{4}+\epsilon_0}.
\end{equation}
We shall find such an $\eta^*$ after a sequence of steps. First let $\epsilon_1, \epsilon_2$ be small positive constants such that 

\begin{equation}\label{condeps}
    \epsilon_1< \frac{\epsilon_0}{4}\quad \quad \text{and} \quad \quad \epsilon_2< \frac{\epsilon_0}{200}.
\end{equation}
Next, let $N$ be the smallest positive integer $N$ such that $4/N<\epsilon_0$. Because we chose the smallest integer, we have $\epsilon_0<2/N$. Further, let $\eta_0$ be a small positive constant such that
\begin{equation}\label{condeta}
    \eta_0< \frac{\epsilon_0\sigma_4}{10^4}.
\end{equation}
The reasons for these conditions will  be clear later in the proof. Next, we iteratively choose a decreasing sequence of small positive constants $\eta_1,\eta_2,\ldots,\eta_N$ such that if you take any consecutive pair $\eta_j, \eta_{j+1}$, then they work as $\epsilon=\eta_j$ and $\eta=\eta_{j+1}$ for both Lemma \ref{mainlem} and Lemma \ref{tertlem}. In particular, we will have $\eta_0\geq \eta_i$ for any $i$. With these constants now fixed, let us begin the induction.

Formally, our inductive statement for $j\in \{1,2,\ldots, N\}$ is the following: there exists $c_j>0$ (independent of $\delta$) such that $V(\eta_j,\delta^{j/N})\geq c_j \delta^{4/N} (\delta^{j/N})^{3/4}$. If we plugged in $j=1$, we get the statement: there exists $c_1>0$ such that 
\begin{equation}\label{bcase}
    V(\eta_1, \delta^{1/N})\geq c_1 (\delta^{1/N})^{4.75}.
\end{equation}
This statement is not hard to prove since the exponent 4.75 is quite high considering we know that Kakeya sets in $\RR^4$ have dimension at least 3. More precisely, using Property (c) of Definition \ref{31} along with \eqref{numtubes}, we get the following: if $(\TT,Y)_{\delta^{1/N}}$ is an $\eta_1$-extremal collection of tubes, then there exists a tube $T\in \TT$ such that $|Y(T)|\gtrsim (\delta^{1/N})^{3+2\eta_1}$. In particular, this proves \eqref{bcase} and therefore, the base case is done.

Now suppose that there exists $c_j>0$ such that $V(\eta_j, \delta^{j/N})\geq c_j \delta^{4/N}(\delta^{j/N})^{3/4}$. We will show that this implies that there exists $c_{j+1}>0$ such that $V(\eta_{j+1},\delta^{\frac{j+1}{N}})\geq c_{j+1} (\delta^{\frac{j+1}{N}})^{3/4}$. Once we show this, the proof of the inductive step will be complete. After $N$ steps in the induction, we will finally prove the statement for $j=N$, that is $V(\eta_N, \delta)\geq c_N \delta^{3/4+4/N}$. This implies \eqref{maineq} with $\eta^*=\eta_N$ and completes the proof of the theorem. Therefore, all that's left to do now is prove the inductive statement.

First, we define some useful notation. We will denote $\Tilde{\delta}=\delta^{\frac{j+1}{N}}$ and $\rho=\delta^{1/N}$. In terms of this notation, we are given $V(\eta_j, \Tilde{\delta}\rho^{-1})\geq c_j \delta^{4/N} (\Tilde{\delta}\rho^{-1})^{3/4}$ and we need to show that $V(\eta_{j+1}, \Tilde{\delta})\geq c_{j+1} \delta^{4/N} \Tilde{\delta}^{3/4}$. 

Let $(\TT,Y)_{\Tilde{\delta}}$ be an $\eta_{j+1}$-extremal collection of $\Tilde{\delta}$-tubes. We now apply Lemma \ref{mainlem} with $\epsilon=\eta_j$ and $\rho=\delta^{1/N}$. We have $\rho\in [\Tilde{\delta}^{1-\eta_j}, \Tilde{\delta}^{\eta_j}]$ since $j+1\geq 2$. Having applied Lemma \ref{mainlem}, we could be in the case (a) or the case (b). Let us suppose we were in case (a) first, then we can use our trilinear Kakeya bound. Specifically, applying Theorem \ref{trikak}  with $\epsilon_1$ in place of $\epsilon$, we get
\begin{align*}
    \bigg| \bigcup_{T\in \TT} Y(T)\bigg|&\geq c_{\epsilon_1} \Tilde{\delta}^{\epsilon_1} (\Tilde{\delta}^{4\eta_{j+1}})^{13/4} \Tilde{\delta}^{3/4} \rho^2 (\Tilde{\delta}^{\eta_{j+1}})^{1/4}\\
    &\geq c_{\epsilon_1}  \delta ^{13\eta_{j+1}+\frac{\eta_{j+1}}{4}+\epsilon_1} \delta^{2/N}\Tilde{\delta}^{3/4}\\
    &\geq c_{\epsilon_1} (\delta^{14\eta_0}) (\delta^{\epsilon_1})\delta^{2/N}\Tilde{\delta}^{3/4}\\
    &\geq c_{\epsilon_1} (\delta^{1/N}) (\delta^{1/N})\delta^{2/N}\Tilde{\delta}^{3/4}.
\end{align*}

In the last line, we applied our conditions \eqref{condeps} and \eqref{condeta}. Thus, we get the bound $\big| \bigcup_{T\in \TT} Y(T)\big|\geq c_{j+1} \delta^{4/N}\Tilde{\delta}^{3/4}$. This finishes the trilinear case (corresponding to Figure \ref{trifig}). 

Next we need to deal with the more complicated ``weakly" plany case (see Figures \ref{wkplnypic} and \ref{rhoplany}) which corresponds to case (b) in Lemma \ref{mainlem}. In this case, we obtain a plany $\eta_j$-extremal collection $(\Tilde{\TT}, \Tilde{Y})_\rho$ at the larger scale $\rho$. By Lemma \ref{tertlem}, we obtain a $\rho^{\eta_j}$-refinement $(\Tilde{\TT}', \Tilde{Y}')_\rho$ of  $(\Tilde{\TT}, \Tilde{Y})_\rho$ which is $(1, \rho^{-4\eta_j})$-two ends and $(\sigma_4, \rho^{-3\eta_{j-1}-3\eta_j})$-robustly transverse. Now we can apply the planebrush bound to the plany $2\eta_j$-extremal collection $(\Tilde{\TT}', \Tilde{Y}')_\rho$. More specifically, we apply Theorem \ref{planebrush} with $\epsilon_2$ in place of $\epsilon$; $1$ in place of $\epsilon_1$; and $\sigma_4$ in place of $\epsilon_2$. Then, we get
\begin{align} 
    \bigg| \bigcup_{\Tilde{T}\in \Tilde{\TT}} \Tilde{Y}(\Tilde{T})\bigg|&\geq c_{\epsilon_2}\rho^{4\eta_j}(\rho^{-3\eta_j-3\eta_{j-1}})^{-2/\sigma_4}\rho^{\epsilon_2} (\rho^{4\eta_j})^{4/3} \rho^{2/3} \rho^{4\eta_j}\nonumber \\
    &\geq c_{\epsilon_2} \rho^{14\eta_j+\frac{6(\eta_j+\eta_{j-1})}{\sigma_4}+\epsilon_2}\rho^{2/3}. \label{ineq1}
\end{align}
We will connect what is happening across different scales using bounds on multiplicities. To this end, let us now work towards translating the above inequality into an upper bound on the average multiplicity $\mu_{\Tilde{Y}}$. We have
\begin{equation}\label{ineq2}
    \mu_{\Tilde{Y}}\bigg| \bigcup_{\Tilde{T}\in \Tilde{\TT}} \Tilde{Y}(\Tilde{T})\bigg|\sim \sum_{\Tilde{T}\in \Tilde{\TT}}|\Tilde{Y}(\Tilde{T})|\lesssim \rho^{-\eta_j}.
\end{equation}
Combining inequalities \eqref{ineq1} and \eqref{ineq2}, we get
\begin{equation}\label{ineq3}
    \mu_{\Tilde{Y}}\leq c_{\epsilon_2} \rho^{-15\eta_j-\frac{6(\eta_j+\eta_{j-1})}{\sigma_4}-\epsilon_2}\rho^{-2/3}.
\end{equation}
Now let us get an upper bound for the multiplicity of $\rho$-tubes through any point. For this, we use the following inequality
$$
\mu_{\Tilde{Y}}=\frac{\sum_{\Tilde{T}\in \Tilde{\TT}}|\Tilde{Y}(\tilde{T})|}{\big| \bigcup_{\tilde{T}\in \Tilde{\TT}} \Tilde{Y}(\tilde{T})\big|}\gtrsim \frac{\rho^{\eta_j}}{\rho^{\sigma_4-\eta_j}}\gtrsim \rho^{-\sigma_4+2\eta_j}.
$$
Note that Property (iv) of Lemma \ref{mainlem} along with the above inequality gives us that $\#\Tilde{\TT}_{\Tilde{Y}}(\Tilde{Q})\lesssim \rho^{-\sigma_4-\epsilon}\lesssim \mu_{\Tilde{Y}}\rho^{-3\eta_j}$. Plugging this into \eqref{ineq3}, we get that
\begin{equation}\label{ineq3a}
    \#\Tilde{\TT}_{\Tilde{Y}}(\Tilde{Q})\leq c_{\epsilon_2} \rho^{-18\eta_j-\frac{6(\eta_j+\eta_{j-1})}{\sigma_4}-\epsilon_2}\rho^{-2/3}.
\end{equation}

At the finer scale, Lemma \ref{mainlem} tells us that the unit rescaling of $(\TT'[\Tilde{T}],Y')_{\Tilde{\delta}}$ relative to $\Tilde{T}$ is an $\eta_j$-extremal collection of $\Tilde{\delta}\rho^{-1}$-tubes and has constant multiplicity $\mu_{\text{fine}}$. Note that $V(\eta_j, \Tilde{\delta}\rho^{-1})$ is in particular a lower bound over the volumes of unions of $\eta_j$-extremal collection of tubes at scale $\Tilde{\delta}\rho^{-1}$. Combining this fact with an inequality like \eqref{ineq2} at the scale $\Tilde{\delta}\rho^{-1}$, we get that
\begin{equation}\label{ineq4}
    \#\TT'[\Tilde{T}](Q)\lesssim \bigg(\frac{\Tilde{\delta}}{\rho}\bigg)^{-\eta_j}V(\eta_j, \Tilde{\delta}\rho^{-1})^{-1}.
\end{equation}
Combining inequalities \eqref{ineq3a} and \eqref{ineq4}, we get
\begin{equation}
    \mu_{Y'}\leq c_{\epsilon_2} \rho^{-18\eta_j-\frac{6(\eta_j+\eta_{j-1})}{\sigma_4}-\epsilon_2}\rho^{-2/3}\bigg(\frac{\Tilde{\delta}}{\rho}\bigg)^{-\eta_j}V(\eta_j, \Tilde{\delta}\rho^{-1})^{-1}.
\end{equation}
Now let us convert this to a lower bound on the volume of union of tubes. We have $\mu_{Y'}\big| \bigcup_{T\in \TT'} Y'(T)\big|\sim \sum_{T\in \TT'}|Y'(T)|\geq \Tilde{\delta}^{24\eta_{j+1}+2\eta_j^3}$ since $(\TT',Y')_{\Tilde{\delta}}$ is a $\Tilde{\delta}^{23\eta_{j+1}+2\eta_j^3}$-refinement of $(\TT,Y)_{\Tilde{\delta}}$. Therefore, linking this with above inequality we get
\begin{align}
    \bigg| \bigcup_{T\in \TT} Y(T)\bigg|&\geq c_{\epsilon_2} \Tilde{\delta}^{24\eta_{j+1}+2\eta_j^3}\rho^{18\eta_j+\frac{6(\eta_j+\eta_{j-1})}{\sigma_4}+\epsilon_2}\rho^{2/3}\bigg(\frac{\Tilde{\delta}}{\rho}\bigg)^{\eta_j} V(\eta_j, \Tilde{\delta}\rho^{-1})\nonumber\\
    &\geq c_{\epsilon_2} \Tilde{\delta}^{24\eta_{j+1}+2\eta_j^3} \rho^{18\eta_j+\frac{6(\eta_j+\eta_{j-1})}{\sigma_4}+\epsilon_2}\rho^{2/3}\bigg(\frac{\Tilde{\delta}}{\rho}\bigg)^{\eta_j} c_j \delta^{4/N} (\Tilde{\delta}\rho^{-1})^{3/4}\nonumber\\
    &\geq c_{\epsilon_2} c_j \bigg(\rho^{-\frac{3}{4}+\frac{2}{3}+17\eta_j+\frac{6(\eta_j+\eta_{j-1})}{\sigma_4}+\epsilon_2} \,\Tilde{\delta}^{24\eta_{j+1}+2\eta_j^3+\eta_j}\bigg) \delta^{4/N}\Tilde{\delta}^{3/4}\nonumber\\
    &\geq c_{\epsilon_2} c_j \bigg(\rho^{-\frac{1}{12}+\frac{23(\eta_j+\eta_{j-1})}{\sigma_4}+\epsilon_2} \,{\delta}^{24\eta_{j+1}+3\eta_j}\bigg) \delta^{4/N}\Tilde{\delta}^{3/4}. \label{brackineq}
\end{align}
We used the inductive hypothesis, namely $V(\eta_j, \Tilde{\delta}\rho^{-1})\geq c_j \delta^{4/N} (\Tilde{\delta}\rho^{-1})^{3/4}$, in the second line of the above series of inequalities. Next we would like to show that the bracketed expression is greater than or equal to 1. Informally, this will happen because the $\rho^{-1/12}$ term in the bracketed expression is large enough to absorb the subsequent terms. We now prove this fact formally. Substituting $\rho=\delta^{1/N}$ into the bracketed expression, we get that
\begin{align*}
    \rho^{-\frac{1}{12}+\frac{23(\eta_j+\eta_{j-1})}{\sigma_4}+\epsilon_2} \,{\delta}^{24\eta_{j+1}+3\eta_j}&= ({\delta^{1/N}})^{-\frac{1}{12}+\frac{23(\eta_j+\eta_{j-1})}{\sigma_4}+\epsilon_2} \,{\delta}^{24\eta_{j+1}+3\eta_j}\\
    &\geq \delta^{[\frac{\epsilon_0}{4}(-\frac{1}{12}+\frac{23(\eta_j+\eta_{j-1})}{\sigma_4}+\epsilon_2)]} {\delta}^{24\eta_{j+1}+3\eta_j}.
\end{align*}
Thus, we have reduced the bracketed expression to a power of $\delta$. Further analyzing the exponent of $\delta$, we get that
\begin{align*}
    \bigg[\frac{\epsilon_0}{4}\bigg(-\frac{1}{12}+\frac{23(\eta_j+\eta_{j-1})}{\sigma_4}+\epsilon_2\bigg)\bigg]+24\eta_{j+1}+3\eta_j&\leq -\frac{\epsilon_0}{50}+\frac{46\eta_0}{\sigma_4}+\epsilon_2+\frac{27\eta_0}{\sigma_4}\\
    &\leq -\bigg(\frac{\epsilon_0}{100}-\epsilon_2\bigg)-\bigg(\frac{\epsilon_0}{100}-\frac{100\eta_0}{\sigma_4}\bigg)
\end{align*}
The first bracketed term is positive because of condition \eqref{condeps} and the second bracketed term is also positive due to the condition \eqref{condeta}. Thus, we have shown that the bracketed expression in \eqref{brackineq} is bounded below by a negative power of $\delta$, in particular it is greater than or equal to 1. Hence, we get that $\big| \bigcup_{T\in \TT} Y(T)\big|\geq c_{j+1} \delta^{4/N}\Tilde{\delta}^{3/4}$. Therefore, in either the trilinear case (a) or plany case (b), we get the required bound. Thus, we can conclude 
\begin{equation}
    V(\eta_{j+1}, \Tilde{\delta})\geq c_{j+1} \delta^{4/N} \Tilde{\delta}^{3/4}.
\end{equation}
This completes the proof of the inductive step and hence the whole theorem.

\end{document}